\documentclass{amsart}

\usepackage{graphicx}         
\usepackage[tight]{subfigure} 
\usepackage{url}              

\newtheorem{theorem}{Theorem}[section]
\newtheorem{corollary}[theorem]{Corollary}
\newtheorem{lemma}[theorem]{Lemma}
\newtheorem{algorithm}[theorem]{Algorithm}
\newtheorem{observation}[theorem]{Observation}
\newtheorem{test}[theorem]{Test}

\theoremstyle{definition}
\newtheorem{defn}[theorem]{Definition}
\newtheorem*{notation}{Notation}

\newcommand{\ab}[1]{#1^\mathrm{ab}}
\newcommand{\abi}[2]{\ab{S_{#1}}(#2)}
\newcommand{\bdry}{\partial}
\newcommand{\gap}{\textit{GAP}}
\newcommand{\hikmot}{\textit{HIKMOT}}
\newcommand{\hypmfd}{cusped finite-volume hyperbolic 3-manifold}
\newcommand{\magma}{\textit{Magma}}
\newcommand{\mb}{M{\"o}bius band}
\newcommand{\mfdb}{M}
\newcommand{\mfdi}{M^\circ}
\newcommand{\N}{\mathbb{N}}

\newcommand{\R}{\mathbb{R}}
\newcommand{\regina}{\textit{Regina}}

\newcommand{\snappea}{\textit{SnapPea}}
\newcommand{\snappy}{\textit{SnapPy}}
\newcommand{\tri}{\mathcal{T}}
\newcommand{\twisted}{\stackrel{\smash{\protect\raisebox{-1mm}[0pt][0pt]%
    {$\scriptstyle\sim$}}}{\times}}
\newcommand{\Z}{\mathbb{Z}}

\begin{document}

\title{The cusped hyperbolic census is complete}
\author{Benjamin A.\ Burton}
\address{School of Mathematics and Physics \\
    The University of Queensland \\
    Brisbane QLD 4072 \\
    Australia}
\email{bab@maths.uq.edu.au}
\thanks{Supported by the Australian Research Council
    under the Discovery Projects funding scheme (project DP1094516).
    Computational resources were provided by the
    Queensland Cyber Infrastructure Foundation.
    The author also offers his warm thanks to the visitors and
    staff at ICERM, Brown University, where much of this work was undertaken.}

\begin{abstract}
    From its creation in 1989 through subsequent extensions,
    the widely-used ``SnapPea census'' now aims to represent all
    cusped finite-volume hyperbolic 3-manifolds
    that can be obtained from $\leq 8$ ideal tetrahedra.
    Its construction, however, has relied on inexact
    computations and some unproven (though reasonable) assumptions,
    and so its completeness was never guaranteed.
    For the first time, we prove here that the census meets its aim:
    we rigorously certify that every ideal 3-manifold
    triangulation with $\leq 8$ tetrahedra is either (i)~homeomorphic to
    one of the census manifolds, or (ii)~non-hyperbolic.

    In addition, we extend the census to 9~tetrahedra, and
    likewise prove this to be complete.
    We also present the first list of all minimal triangulations of
    all census manifolds, including non-geometric as well as
    geometric triangulations.
\end{abstract}

\maketitle

%
%

\section{Introduction}

Over its quarter-century history,
the ``SnapPea census'' of cusped finite-volume hyperbolic 3-manifolds
has been an invaluable resource for low-dimensional topologists.
In its modern form it contains
21\,918 cusped 3-manifolds\footnote{%
    The original census listed 21\,919 manifolds,
    but two were later found to be homeomorphic~\cite{burton13-duplicate}.},
believed to represent all cusped finite-volume hyperbolic 3-manifolds
that can be built from $n \leq 8$ ideal tetrahedra.
The original census was created in 1989 by Hildebrand and Weeks for $n \leq 5$
\cite{hildebrand89-cuspedcensusold}, and was later expanded by
Callahan, Hildebrand and Weeks for $n=6,7$ and Thistlethwaite
for $n=8$ \cite{callahan99-cuspedcensus,thistlethwaite10-cusped8}.
Portions of the census are now shipped with topological software
packages such as {\snappy} \cite{snappy} and {\regina} \cite{regina}.

Despite its long history, however, questions of accuracy remain unresolved.
The key issues are that
(i)~those manifolds included in the census are only
those for which the software {\snappea} \cite{snappea} identifies a
\emph{geometric triangulation}---one that decomposes the manifold into
positive-volume ideal hyperbolic tetrahedra; and that
(ii)~{\snappea} uses floating point arithmetic to test whether a
triangulation is geometric.
In theory, this allows for several types of error:
\begin{enumerate}
    \item \label{en-err-falsepos}
    {\snappea} might incorrectly identify a non-geometric triangulation as
    geometric, due to numerical approximation errors;
    \item \label{en-err-falseneg}
    {\snappea} might incorrectly identify a geometric triangulation as
    non-geo\-met\-ric, due to either approximation errors or numerical
    instability (where successive approximations to a geometric
    structure fail to converge);
    \item \label{en-err-nongeom}
    {\snappea} might fail to identify a manifold as hyperbolic
    because all of its
    triangulations with $\leq n$ tetrahedra are non-geometric.
\end{enumerate}

Issue~(\ref{en-err-falsepos}) could lead to false positives.
This possibility can now be eliminated using the techniques of
Moser \cite{moser09-proving} and Hoffman et~al.\ \cite{hoffman13-hikmot},
who use numerical methods to show that {\snappea}'s \emph{approximate}
geometric structure is indeed an approximation to an \emph{exact}
geometric structure.
In particular, Moser has shown that the $n \leq 7$-tetrahedron census
contains no false positives, and Hoffman et~al.\ have shown that the
orientable $n \leq 8$-tetrahedron census contains no false positives.

Issues~(\ref{en-err-falseneg}) and~(\ref{en-err-nongeom})
could lead to false negatives.  Numerical methods alone cannot solve
this: even if we could prove conclusively that a triangulation is
non-geometric (which the numerical methods above cannot),
issue~(\ref{en-err-nongeom}) means that we could still fail to identify
a manifold as hyperbolic because \emph{all} of its
$\leq n$-tetrahedron triangulations
are non-geometric.  Indeed, it is still not known whether there might
exist cusped hyperbolic 3-manifolds with no geometric triangulations at all.

Our main result here is to resolve
issues~(\ref{en-err-falseneg}) and~(\ref{en-err-nongeom}),
and thus---after 25~years---rigorously prove
that the SnapPea census has no false negatives.  Specifically:

\begin{theorem} \label{t-falseneg}
    Every ideal 3-manifold triangulation with $n \leq 8$ tetrahedra is either
    (i)~homeomorphic to one of the manifolds in the
    Callahan-Hildebrand-Thistlethwaite-Weeks census tables, or
    (ii)~certified to represent a non-hyperbolic manifold.
\end{theorem}

The proof, which is computationally intensive and algorithmically
non-trivial, involves two major stages.  The first is to enumerate all
ideal 3-manifold triangulations with $\leq 8$ tetrahedra,
under several combinatorial constraints that we prove in Section~\ref{s-min}.
The second is to certify that every one of the resulting manifolds
either matches one of the Callahan-Hildebrand-Thistlethwaite-Weeks
census manifolds or is non-hyperbolic.  All computations are
exact, thus avoiding numerical errors.
We discuss details of these two stages
in Sections~\ref{s-gen} and~\ref{s-process} respectively.

By combining Theorem~\ref{t-falseneg} with the ``no false positives''
results of Moser \cite{moser09-proving} and
Hoffman et~al.\ \cite{hoffman13-hikmot}, and by running the latter
authors' software {\hikmot} over the non-orientable 8-tetrahedron census
(which neither paper \cite{hoffman13-hikmot,moser09-proving} examines),
we can finally show that the SnapPea census meets its original aim:

\begin{corollary} \label{c-complete8}
    The Callahan-Hildebrand-Thistlethwaite-Weeks census tables
    exactly represent all cusped finite-volume hyperbolic 3-manifolds
    that can be constructed from $n \leq 8$ ideal tetrahedra,
    with no intruders (false positives) and no omissions (false negatives).
\end{corollary}

\begin{table}[tb]
\caption{Summary of census data} \label{tab-census}
\begin{tabular}{c|rrr|rrr}
    & \multicolumn{3}{c|}{\bf Manifolds} &
          \multicolumn{3}{c}{\bf Minimal triangulations} \\
    Tetrahedra & Orbl & Non-orbl & Total & Orbl & Non-orbl & Total \\
    \hline
    1 & 0 & 1 & 1                  & 0 & 1 & 1 \\
    2 & 2 & 2 & 4                  & 2 & 3 & 5 \\
    3 & 9 & 7 & 16                 & 10 & 11 & 21 \\
    4 & 56 & 26 & 82               & 75 & 60 & 135 \\
    5 & 234 & 78 & 312             & 360 & 179 & 539 \\
    6 & 962 & 258 & 1\,220         & 1\,736 & 801 & 2\,537 \\
    7 & 3\,552 & 887 & 4\,439      & 7\,413 & 3\,202 & 10\,615 \\
    8 & 12\,846 & 2\,998 & 15\,844 & 30\,450 & 12\,777 & 43\,227 \\
    9 & 44\,250 & 9\,788 & 54\,038 & 122\,136 & 49\,896 & 172\,032 \\
    \hline
    Total & 61\,911 & 14\,045 & 75\,956 & 162\,182 & 66\,930 & 229\,112
\end{tabular}
\end{table}

In addition, we use our techniques to extend the census beyond its
current limits, up to $n \leq 9$ tetrahedra.
This involves enumerating and processing
more than $1.5$~million ideal triangulations
in total, yielding a final list of 75\,956 census manifolds.
Table~\ref{tab-census} summarises the new census data, and
Section~\ref{s-census} describes how the final computations tie together
our various theoretical results.
This 9-tetrahedron census
is likewise rigorously guarded against both false positives
(using {\hikmot}) and false negatives (using the techniques in this
paper), and so we obtain:

\begin{theorem} \label{t-complete9}
    The new census data outlined in Table~\ref{tab-census}
    exactly represents all cusped finite-volume hyperbolic 3-manifolds
    that can be constructed from $n \leq 9$ ideal tetrahedra,
    with no intruders and no omissions.
\end{theorem}

In Section~\ref{s-dup} we resolve the issue of \emph{duplicates}.
That is, we prove that no manifold appears more than once in the
final census tables.  This is a real possibility---indeed, the
original $n \leq 7$-tetrahedron census contained a duplicate pair
that was not caught until 14~years after that census was published
\cite{burton13-duplicate}.
In this paper we use standard group-theoretical techniques to show:

\begin{theorem} \label{t-dup}
    No two of the $75\,956$ manifolds from Table~\ref{tab-census}
    are homeomorphic.
\end{theorem}

Again this result relies on exact computation; in particular,
inexact floating-point invariants such as
hyperbolic volume and shortest geodesic are not used.
A consequence of this result is that all $229\,112$ triangulations
in our census are proven to be minimal, and so we have:

\begin{corollary} \label{c-min}
    The $229\,112$ triangulations from Table~\ref{tab-census}
    are precisely all minimal ideal triangulations of all
    {\hypmfd}s with $n \leq 9$ tetrahedra.
\end{corollary}

This gives us the first comprehensive database of
\emph{all minimal triangulations} of all census manifolds,
including both geometric and non-geometric triangulations.
One immediate application would be in studying the conjecture that
every cusped hyperbolic 3-manifold has a geometric triangulation.

As a final note: In the late 1980s,
Adams and Sherman studied the minimum number of
ideal tetrahedra required to build a $k$-cusped hyperbolic 3-manifold
\cite{adams91-ideal}.  One of their results was the following:

\begin{theorem}[Adams and Sherman \cite{adams91-ideal}]
    The smallest number of ideal tetrahedra required to build a 5-cusped
    finite-volume hyperbolic manifold is $\sigma_{10}=10$.
\end{theorem}

The proof was never given in full, since the detailed argument that no such
manifold exists for $n=9$ tetrahedra remains in an unpublished thesis of
Sherman \cite{sherman88-thesis}.
The results of this paper yield an alternative rigorous
computer proof.

The full database of all $75\,956$ census manifolds and all $229\,112$
minimal triangulations can be downloaded from the website
\url{http://www.maths.uq.edu.au/~bab/code/}, and will be included in the
coming release of \regina~4.96.

Most computations from this paper are performed using the software package
{\regina} \cite{burton04-regina,regina}; some computations also use
{\hikmot} \cite{hoffman13-hikmot}, {\magma} \cite{bosma97-magma}
and {\snappy} \cite{snappy}, and these are noted where they occur.

The author thanks the residents at ICERM during the 2013 fall semester,
and Saul Schleimer and Stephan Tillmann in particular,
for many stimulating discussions during the development of this work.

%
%

\section{Preliminaries} \label{s-prelim}

We begin in Sections~\ref{s-prelim-tri}--\ref{s-prelim-normal}
with essential facts about triangulations, hyperbolic manifolds
and normal surfaces.
Following this, Sections~\ref{s-prelim-barrier} and \ref{s-prelim-crushing}
outline two specialised techniques that play an important role in this
paper: \emph{barrier surfaces}, and \emph{crushing a normal surface}.
Both were developed by Jaco and Rubinstein to support their theory of
0-efficiency \cite{jaco03-0-efficiency};
here we present only what is required for this paper, and
we refer the reader to Jaco and Rubinstein's original paper for further
details.

All 3-manifolds in this paper are connected unless otherwise noted.
If $\mfdb$ is a compact 3-manifold with boundary, we let
$\mfdi$ denote its non-compact interior.  We explicitly note that there
are no restrictions on orientability in this paper (i.e., 3-manifolds
may be either orientable or non-orientable).


\subsection{Triangulations} \label{s-prelim-tri}

A \emph{generalised triangulation} $\tri$ is a collection of
$n$ abstract tetrahedra, some or all of whose triangular faces are affinely
identified or ``glued together'' in pairs.  The result need not be a
simplicial complex; in particular, we allow two faces of the same
tetrahedron to be identified, and we allow two tetrahedra to be
glued together along multiple pairs of faces.
Generalised triangulations may be disconnected (or even empty if $n=0$).
Each tetrahedron face that is not identified with some other face is
called a \emph{boundary triangle} of $\tri$.

These face identifications induce identifications between the edges of
tetrahedra, and each resulting class of identified edges is called a single
\emph{edge of the triangulation} $\tri$.
We define a \emph{vertex of the triangulation} $\tri$ similarly.
An \emph{invalid edge} is an edge of $\tri$ that (as a result of the
face identifications) is identified with itself in reverse.
It is common to find \emph{one-vertex triangulations}, in which
all $4n$ tetrahedron vertices are identified to a single point in $\tri$.

We interpret $\tri$ as a topological space using the identification topology.
If $v$ is a vertex of $\tri$, then the \emph{link} of $v$ is the
frontier of a small regular neighbourhood of $v$ in $\tri$.
If the link of $v$ is a disc then we call $v$ a \emph{boundary vertex},
and if the link of $v$ is a sphere then we call $v$ an \emph{internal vertex}.
If the link of $v$ is some other closed surface, then we call $v$ an
\emph{ideal vertex}.

If $\mfdb$ is a compact 3-manifold (with or without boundary) and
$\tri$ is homeomorphic to $\mfdb$, then we
say that $\tri$ is a \emph{triangulation of $\mfdb$}.
In this case, $\tri$ must have no invalid edges, and every vertex
link must be a disc or sphere.
Note that $\partial \mfdb$ is formed from the boundary triangles of $\tri$.

If $\mfdb$ is a compact 3-manifold with boundary, and if
$\tri$ becomes homeomorphic to the non-compact interior $\mfdi$
once its ideal vertices are removed,
then we say that $\tri$ is an \emph{ideal triangulation of $\mfdi$}.
In this case, $\tri$ must have no invalid edges and no boundary triangles,
and at least one vertex must be ideal (though there may be internal
vertices also).\footnote{%
    Many authors take a more restricted definition of
    ideal triangulations, in which all vertex links must be ideal.
    Although our definition is more liberal (by allowing internal vertices),
    we show in Theorem~\ref{t-nointernal} that for \emph{minimal}
    triangulations of {\hypmfd}s, both of these definitions coincide.}
We say that $\tri$ is a \emph{minimal ideal triangulation of $\mfdi$}
if there is no ideal triangulation of $\mfdi$ with fewer tetrahedra.

Given an ideal triangulation $\tri$ of $\mfdi$ as described
above, we can build a triangulation of the corresponding
compact manifold $\mfdb$ by \emph{truncating} the ideal vertices of
$\tri$.  This is a messy but straightforward procedure, in which we cut out
a small neighbourhood of every ideal vertex from every tetrahedron of
$\tri$, and then retriangulate the resulting truncated tetrahedra.

\emph{Pachner moves} \cite{pachner91-moves}
(also known as \emph{bistellar flips}) are local
combinatorial operations on triangulations: for any triangulation $\tri$
(either ideal or not), applying a Pachner move will result in a new
triangulation of the same manifold.
The most important Pachner moves for this
paper are the \emph{2-3 move}, in which two distinct tetrahedra joined
along a common triangle are replaced by three distinct tetrahedra surrounding
a common edge, and the \emph{3-2 move}, which is the inverse operation.


\subsection{Hyperbolic manifolds} \label{s-prelim-hyp}

Let $\mfdb$ be a compact 3-manifold with boundary whose interior
$\mfdi$ is a {\hypmfd}.

Then each boundary component of $\mfdb$ must be a torus or Klein bottle.
Moreover, {\mfdb} cannot contain any properly embedded surfaces that are:
\begin{itemize}
    \item \emph{essential spheres}, i.e., spheres that do not bound
    balls;
    \item \emph{projective planes} of any type;
    \item \emph{essential compression discs}, i.e., discs in $\mfdb$
    whose boundaries do not bound discs in $\bdry \mfdb$;
    \item \emph{essential tori}, i.e., tori that are
    $\pi_1$-injective and not homotopic into $\bdry \mfdb$;
    \item \emph{essential annuli}, i.e., annuli
    that are $\pi_1$-injective and not properly homotopic into $\bdry \mfdb$.
\end{itemize}
More strongly, any properly embedded torus or annulus in $\mfdb$
that is $\pi_1$-injective must be \emph{boundary-parallel}, i.e.,
properly isotopic into $\bdry \mfdb$.
See \cite{bonahon02-handbook,kapovich09-hyperbolic}
for detailed discussions on such results that cover both
orientable and non-orientable manifolds.

%
%
%
%

A \emph{geometric triangulation} of $\mfdi$ is an ideal triangulation
of $\mfdi$ where, in the context of a complete hyperbolic structure on $\mfdi$,
all vertices are ideal,
all edges are geodesics,
all triangular faces are portions of geodesic planes, and
every tetrahedron is positively oriented (i.e., has positive hyperbolic volume).
The software {\snappea} \cite{snappea} (and its successor
{\snappy} \cite{snappy}) can be used to test whether a given
triangulation $\tri$ is geometric and, if so, use $\tri$ to describe
the complete hyperbolic structure on the underlying manifold
(all subject to floating point approximations).

Any {\hypmfd} has a canonical \emph{Epstein-Penner cell decomposition}
\cite{epstein88-euclidean}:
two such 3-manifolds are homeomorphic if and only if their
Epstein-Penner decompositions are combinatorially isomorphic.
Given a triangulation $\tri$ that {\snappea} believes is geometric,
{\snappea} can attempt to compute the Epstein-Penner decomposition
for the underlying manifold \cite{weeks93-convex}.
Although numerical errors might cause {\snappea} to obtain the wrong
cell decomposition
\cite{burton13-duplicate}, the underlying algorithm is based on Pachner
moves, and so it is guaranteed that whatever cell decomposition it
\emph{does} compute will represent the same manifold as the
original triangulation $\tri$.


\subsection{Normal surfaces} \label{s-prelim-normal}

Let $\tri$ be a triangulation of a 3-manifold
($\tri$ may be ideal or non-ideal).
A \emph{normal surface} in $\tri$ is a properly embedded
surface $S$ that meets each tetrahedron of $\tri$ in a (possibly empty)
collection of curvilinear triangles and/or quadrilaterals, as
illustrated in Figure~\ref{fig-normaldiscs}.  We explicitly allow normal
surfaces to be disconnected, or even empty.  We do, however, insist in
this paper that a normal surface contains finitely many triangles and
quadrilaterals (i.e., we do not allow the non-compact
\emph{spun-normal surfaces} that can appear in ideal triangulations
\cite{tillmann08-finite}).
Two normal surfaces are \emph{normally isotopic} if they are related by
an ambient isotopy of $\tri$ that preserves each simplex of $\tri$.

\begin{figure}[tb]
    \centering
    \includegraphics[scale=0.6]{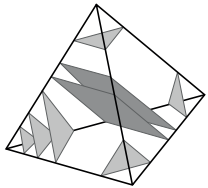}
    \caption{Normal triangles and quadrilaterals in a tetrahedron}
    \label{fig-normaldiscs}
\end{figure}

For any vertex $v$ of the triangulation $\tri$, the link of $v$ can be
expressed as a normal surface containing only triangles (i.e., no
quadrilaterals).  If we say that a normal surface $S$ is the
link of $v$, then we mean more precisely that $S$ is of this form.
Note that, as a normal surface, the link of $v$ is unique up to normal isotopy.
More generally, we say that a normal surface $S$ is \emph{vertex linking}
if it contains only triangles, or equivalently if $S$ is a
(possibly empty) union of vertex links.

If $\tri$ contains $n$ tetrahedra, then a normal surface in $\tri$ can
be specified by a non-negative vector in $\Z^{7n}$
(called \emph{standard coordinates}), or by a non-negative vector in
$\Z^{3n}$ (called \emph{quadrilateral coordinates}).
The vector in standard coordinates defines the surface up to normal
isotopy, and the vector in quadrilateral coordinates defines the surface
up to normal isotopy and addition\,/\,subtraction of vertex links.
In each coordinate system we identify a finite ``basis'' of
normal surfaces from which all others can be generated; essentially
these correspond to extreme rays of a polyhedral cone
\cite{jaco84-haken}.
The basis surfaces in each coordinate system are called
\emph{standard vertex normal surfaces} and
\emph{quadrilateral vertex normal surfaces} respectively.

For compact manifolds, the quadrilateral vertex normal surfaces are
a strict subset of the standard vertex normal surfaces, are typically
much faster to compute, and in many settings contain representatives of those
surfaces that are ``topologically interesting''.  In ideal triangulations,
the quadrilateral vertex normal surfaces can be much slower to compute,
there may be many more of them, and they often contain
non-compact \emph{spun-normal surfaces} with infinitely many
triangles, which we explicitly disallow in this paper.
See \cite{burton09-convert} for further details on the combinatorial and
computational relationships between the two systems.


\subsection{Barrier surfaces} \label{s-prelim-barrier}

Let $\mfdb$ be a compact 3-manifold with boundary,
and let $\tri$ be an ideal triangulation of the non-compact interior $\mfdi$.
Given any embedded closed surface $S \subset \mfdi$, there is
a well-known \emph{normalisation} process that converts $S$ into a
normal surface $N \subset \tri$.

The normal surface $N$ is obtained from $S$ by a series of
isotopies, compressions, and deletion of trivial sphere components.
The compressions may or may not be trivial (i.e., we might compress along
curves that are trivial in the surface).
Any sphere components that are deleted must be trivial (i.e., must
bound a ball in $\mfdi$).
The resulting normal surface $N$ might be disconnected, and might even
be empty.
See \cite[Section~3.1]{jaco03-0-efficiency} for a more detailed summary
of the normalisation process.

Normalisation can, in some cases, make widespread changes to the
original surface $S$.
The barrier surface technology of Jaco and Rubinstein allows us to
limit the scope of these changes, and thus retain more precise control
over the relationship between $S$ and $N$.  Here we outline just
those parts of the theory that we need here; for the full theory the
reader is referred to \cite[Section~3.2]{jaco03-0-efficiency}.

Let $B$ be an embedded closed surface in $\tri$, and let $C$ be some
connected component of the complement $\mfdi \backslash B$.
We say that $B$ is a \emph{barrier surface for $C$} if any embedded
closed surface in $C$ can be normalised entirely within $C$.  In other words,
when we normalise any closed surface $S \subset C$,
the normalisation process never isotopes the surface past the ``barrier'' $B$,
and never compresses along a disc that cuts through $B$.

Often the component $C$ of $\mfdi \backslash B$ is clear from context
(e.g., because it contains the surface $S$ that we are attempting to
normalise).
In this case we simply say that $B$ is a \emph{barrier surface}.
Amongst other examples, Jaco and Rubinstein show that all of the
following are barrier surfaces \cite[Theorem~3.2]{jaco03-0-efficiency}:
\begin{itemize}
    \item \emph{The boundary $B$ of a small regular neighbourhood of a
    subcomplex $\mathcal{K}$ of $\tri$.}
    Here the component $C$ of $\mfdi \backslash B$ must be some
    component not meeting the subcomplex $\mathcal{K}$.
    We often abuse terminology here and simply refer to
    $\mathcal{K}$ itself as a ``barrier to normalisation''.
    Important examples are where $\mathcal{K}$ is a vertex of $\tri$, or
    a single edge of $\tri$.

    \item \emph{The boundary $B$ of a small regular neighbourhood of
    a normal surface $F \subset \tri$.}
    Likewise, the component $C$ of $\mfdi \backslash B$ must be some
    component not meeting the normal surface $F$, and we often simply
    refer to $F$ itself as a ``barrier to normalisation''.

    \item \emph{A combination of the two cases above.}
    Specifically, let $F \subset \tri$ be a normal surface,
    let $\tri'$ be the cell decomposition induced by
    splitting $\tri$ along $F$ (so $\tri$ may contain truncated tetrahedra,
    triangular or quadrilateral prisms, and/or other non-tetrahedron pieces),
    and let $\mathcal{K}$ be a subcomplex of $\tri'$.
    Then the boundary $B$ of a small regular neighbourhood of
    $F \cup \mathcal{K}$ is a barrier surface.
    Once more the component $C$ of $\mfdi \backslash B$ must
    not meet $F \cup \mathcal{K}$, and we often
    refer to $F \cup \mathcal{K}$ itself as a ``barrier to normalisation''.

    An example of such a barrier in this paper appears in the proof of
    Theorem~\ref{t-nospheres}, where $F$ is a normal sphere and
    $\mathcal{K}$ is a fragment of an edge of $\tri$ that joins $F$
    to a vertex of $\tri$.
    A more complex example appears in the proof of
    Lemma~\ref{l-crushsame}, where the subcomplex $\mathcal{K}$ is an
    annulus or {\mb} embedded in the 2-skeleton of $\tri$ whose
    boundary runs along the normal surface $F$.
\end{itemize}


\subsection{Crushing normal surfaces} \label{s-prelim-crushing}

Many topological algorithms require us to cut a triangulation open along a
normal surface.  The problem with this operation is that it can be
extremely expensive: the number of tetrahedra in the triangulation may
grow exponentially as a result.

Jaco and Rubinstein introduce an alternative operation, called
\emph{crushing} \cite[Section~4]{jaco03-0-efficiency}.
This has the advantage that the number of tetrahedra
never increases (and indeed, strictly decreases if the surface is
non-trivial).  The disadvantage, however, is that the crushing operation
can have unintended topological side-effects.  Here we give a very brief
outline of the operation and its effects in our setting;
see \cite{jaco03-0-efficiency} for full details or
\cite{burton12-crushing-dcg} for a simplified treatment.

Let $\mfdb$ be a compact 3-manifold with boundary,
let $\tri$ be an ideal triangulation of the non-compact interior $\mfdi$,
and let $S$ be a two-sided normal surface in $\tri$.
To \emph{crush $S$ in $\tri$}, we perform the following steps:

\begin{figure}[tb]
    \centering
    \includegraphics[scale=0.6]{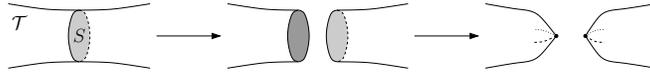}
    \caption{Cutting along $S$ and collapsing the resulting boundaries}
    \label{fig-crushbdry}
\end{figure}

\begin{figure}[tb]
    \centering
    \includegraphics[scale=0.8]{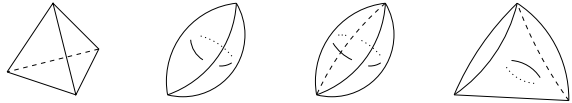}
    \caption{Different types of cells obtained after collapsing}
    \label{fig-crushpieces}
\end{figure}

\begin{enumerate}
    \item \label{en-crush-decomp}
    We cut $\tri$ open along the surface $S$
    and collapse the two resulting boundary components to points
    (Figure~\ref{fig-crushbdry}).
    This splits each tetrahedron of $\tri$
    into a collection of cells, each of which
    is either a tetrahedron, a three-sided football,
    a four-sided football, or a
    triangular purse (Figure~\ref{fig-crushpieces}).

    \begin{figure}[tb]
        \centering
        \includegraphics[scale=0.8]{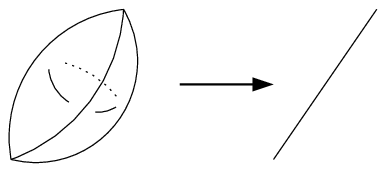}
        \hspace{0.8cm}
        \includegraphics[scale=0.8]{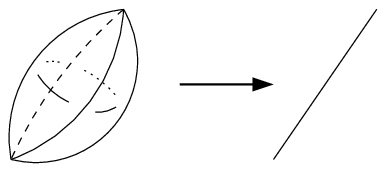}
        \hspace{0.8cm}
        \includegraphics[scale=0.8]{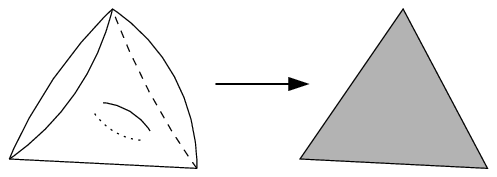}
        \caption{Flattening non-tetrahedron cells}
        \label{fig-crushflatten}
    \end{figure}

    \begin{figure}[tb]
        \centering
        \subfigure[Dangling edges and 2-faces are removed%
            \label{fig-clean-edgeface}]{
            \includegraphics[scale=0.47]{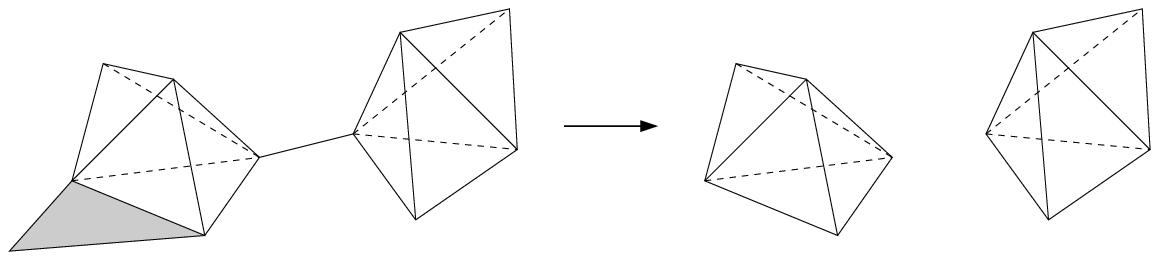}}
        \qquad
        \subfigure[Pinched tetrahedra fall apart%
            \label{fig-clean-pinch}]{
            \includegraphics[scale=0.47]{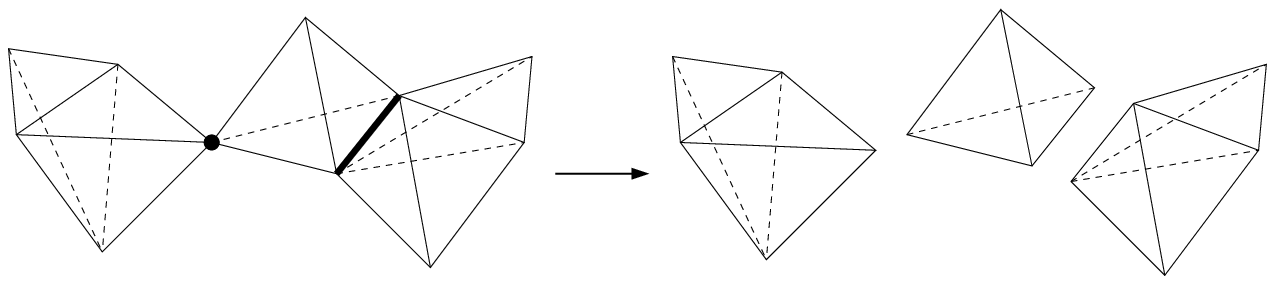}}
        \caption{Cleaning up after crushing}
        \label{fig-clean}
    \end{figure}

    \item \label{en-crush-flatten}
    We simultaneously flatten all footballs to edges
    and all triangular purses to triangular faces
    (Figure~\ref{fig-crushflatten}).
    Any ``dangling'' edges or 2-faces that do not belong to any
    tetrahedra are simply removed (Figure~\ref{fig-clean-edgeface}),
    and any tetrahedra that are
    ``pinched'' along edges or vertices simply fall apart
    (Figure~\ref{fig-clean-pinch}).
\end{enumerate}

The result will be a new generalised triangulation (possibly
disconnected or possibly even empty), and the topological type of
this triangulation is unclear:
\begin{itemize}
    \item The topological effect of
    step~(\ref{en-crush-decomp}) is simple to analyse.
    If $S$ is a sphere, then step~(\ref{en-crush-decomp}) effectively
    cuts along $S$ and fills the two new boundary spheres with balls.
    Otherwise step~(\ref{en-crush-decomp}) cuts along
    $S$ and converts the two new boundary components to ideal vertices,
    effectively producing an ``ideal cell decomposition'' of the
    non-compact manifold $\mfdi \backslash S$.

    \item Step~(\ref{en-crush-flatten}) is more problematic.
    In general, flattening footballs and triangular purses can further
    change the topology, and might even introduce invalid edges.
    Although the ``damage'' can be contained in some special cases
    (such as crushing discs or spheres in compact manifolds
    \cite{burton12-crushing-dcg,jaco03-0-efficiency}),
    in general one must be very careful about drawing any
    conclusions about the topology of the final triangulation.
\end{itemize}

\begin{figure}[tb]
\centering
\subfigure[Flattening a triangular pillow]{\label{fig-jrtripillow}%
    \includegraphics[scale=0.75]{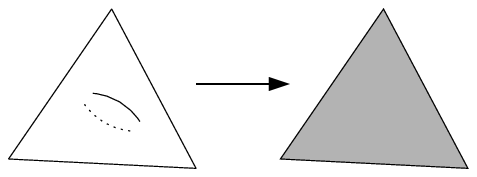}}
\hspace{0.45cm}
\subfigure[Flattening a bigonal pillow]{\label{fig-jrbipillow}%
    ~\includegraphics[scale=0.75]{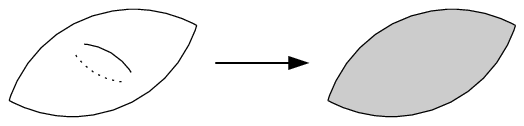}~}
\hspace{0.45cm}
\subfigure[Flattening a bigon]{\label{fig-jrbigon}%
    \includegraphics[scale=0.75]{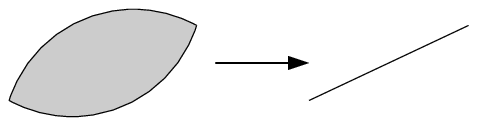}}
\caption{The three atomic operations}
\label{fig-atomic}
\end{figure}

To help understand the potential effect of step~(\ref{en-crush-flatten}),
we can use the \emph{crushing lemma} \cite{burton12-crushing-dcg}.
The crushing lemma shows that, instead of \emph{simultaneously} flattening all
footballs and purses in step~(\ref{en-crush-flatten}), we can replace
this with a \emph{sequence} of zero or more of the following ``atomic
operations'', all illustrated in Figure~\ref{fig-atomic}:
(i)~flattening triangular pillows to triangular faces;
(ii)~flattening bigonal pillows to bigon faces; and
(iii)~flattening bigon faces to edges.
As before, we remove dangling edges or 2-faces and allow pinched
tetrahedra to fall apart, as seen in Figure~\ref{fig-clean}.
The result is that we can analyse the topological effect of
step~(\ref{en-crush-flatten}) inductively,
simply by studying the possible effect of each individual atomic move.

A final observation is that each tetrahedron of the
original triangulation $\tri$ gives rise to at most one tetrahedron
in the final triangulation after crushing, and so the total number of
tetrahedra does not increase.
More precisely, the tetrahedra of $\tri$ that ``survive'' through to
the final triangulation are precisely those tetrahedra that meet the normal
surface $S$ only in triangles (i.e., no quadrilaterals).
In particular, if $S$ is not vertex linking (i.e., $S$ contains at least
one quadrilateral piece),
then the final triangulation will contain
\emph{strictly fewer tetrahedra} than $\tri$.

%
%

\section{Minimal triangulations of cusped hyperbolic manifolds} \label{s-min}

Here we prove a series of simple combinatorial conditions that must be
satisfied by any minimal ideal triangulation of a \hypmfd:
there can be no internal vertices (Theorem~\ref{t-nointernal}),
no normal 2-spheres (Theorem~\ref{t-nospheres}),
only limited normal tori or Klein bottles (Theorem~\ref{t-notorikb}),
and no low-degree edges (Theorem~\ref{t-lowdeg}).

Orientable variants of the first two results
(no internal vertices or normal 2-spheres)
were proven by Jaco and Rubinstein \cite{jaco03-0-efficiency};
here we extend these to the non-orientable setting.

The fourth result (no low-degree edges) is claimed by Hildebrand and Weeks
but without proof \cite{hildebrand89-cuspedcensusold}.
Although this is easily shown for geometric triangulations (the focus of
the Hildebrand-Weeks census),
significant complications arise in the general case that are not
resolved in the literature.
We give a full proof here.

Recall from Section~\ref{s-prelim} that,
if $\mfdi$ is a {\hypmfd}, then
any properly embedded annulus in the compact manifold $\mfdb$
that is $\pi_1$-injective must be boundary-parallel.
We begin by recasting this fact into a more convenient form.

\begin{observation} \label{obs-am}
    Let $\mfdb$ be a compact 3-manifold with boundary whose interior
    $\mfdi$ is a {\hypmfd}.  Then:
    \begin{itemize}
        \item any properly embedded annulus in $\mfdb$ must be
        two-sided, and must either
        be boundary-parallel or have boundary curves
        that are both trivial in $\partial M$;
        \item any properly embedded {\mb} in $\mfdb$ must be two-sided
        and boundary-parallel.
    \end{itemize}
\end{observation}

\begin{proof}
    First we note that any boundary-parallel surface must be two-sided.
    Now suppose that $S \subset \mfdb$ is a properly embedded annulus or {\mb}
    that is not boundary-parallel.

    If $S$ is an annulus, then we know from Section~\ref{s-prelim-hyp} that
    $S$ cannot be $\pi_1$-injective.  Therefore the two curves
    of $\bdry S$ are trivial in $M$, and since $M$ has no essential
    compression discs they must be trivial in $\bdry \mfdb$ as well.
    Moreover, it follows from this
    that $\bdry S$ must be two-sided in $\bdry \mfdb$,
    and thus $S$ is two-sided in $M$.

    If $S$ is a one-sided {\mb},
    let $T$ be the double of $S$ (i.e., the frontier of a
    regular neighbourhood of $S$ in $M$).
    Then $T$ is a two-sided annulus, and
    by the argument above either
    (i)~$T$ is boundary-parallel, or
    (ii)~the curves of $\bdry T$ are trivial in $\bdry \mfdb$.
    Case~(i) implies that $\mfdb$ is a twisted $I$-bundle over the
    {\mb} (i.e., a solid torus), contradicting our assumptions on $\mfdb$.
    In case~(ii), each boundary curve of $\bdry T$ is parallel to
    $\bdry S$ in $\bdry \mfdb$; therefore $\bdry S$ bounds a disc
    $D \subset \bdry M$ and $\mfdb$ contains an embedded projective plane
    $D \cup S$, again a contradiction.

    If $S$ is a two-sided {\mb}, then $M$ must be
    non-orientable.  Here we apply our earlier results
    to the orientable double cover $\tilde{M}$ of $M$,
    where $S$ lifts to a two-sided annulus $\tilde{S}$.
    Either (i)~$\tilde{S}$ is boundary-parallel in $\tilde{M}$, in which
    case $S$ is boundary-parallel in $\mfdb$; or
    (ii)~the curves of $\bdry \tilde{S}$ are trivial in $\bdry \tilde{M}$,
    in which case $\bdry S$ bounds a disc in $\bdry M$ and we obtain an
    embedded projective plane as before.
\end{proof}

For the next result we require the notion of an
\emph{outermost} normal surface:

\begin{defn}
    Let $\tri$ be an ideal 3-manifold triangulation,
    let $v$ be an ideal vertex of $\tri$,
    and $S$ be a normal surface in $\tri$ that is isotopic to the
    boundary of a small regular neighbourhood of $v$.
    We refer to such a surface as
    \emph{boundary-parallel to $v$}.

    We say that $S$ is \emph{outermost} with respect to $v$ if,
    for any normal surface $S'$ that is isotopic to $S$ and
    disjoint from $S$, either $S'$ lies inside the collar
    between $S$ and a small neighbourhood of $v$, or else
    $S'$ is \emph{normally} isotopic to $S$.
\end{defn}

Essentially, being outermost means that the only isotopic normal surfaces
strictly further away from $v$ are ``copies'' of the same
normal surface, formed from the same combination of triangles
and/or quadrilaterals.

The following lemma is our main helper tool for proving properties
of minimal triangulations.  Note that Jaco and Rubinstein use a related
technique in the orientable setting; see in particular
\cite[Theorem~7.4]{jaco03-0-efficiency}.

\begin{lemma} \label{l-crushsame}
    Let $\mfdb$ be a compact 3-manifold with boundary whose interior
    $\mfdi$ is a {\hypmfd}.
    Let $\tri$ be an ideal triangulation of $\mfdi$,
    and let $v$ be an ideal vertex of $\tri$.
    If $S$ is a normal surface in $\tri$ that is boundary-parallel onto $v$
    and outermost with respect to $v$,
    then if we crush $S$ as described in Section~\ref{s-prelim-crushing},
    some component of the resulting triangulation will also be an ideal
    triangulation of $\mfdi$.
\end{lemma}

\begin{proof}
    Recall from
    Section~\ref{s-prelim-crushing} that, by the crushing lemma,
    the full crushing process can be realised via the following steps:
    \begin{enumerate}
        \item cutting $\tri$ open along $S$ and then collapsing both copies
        of $S$ on the boundary to points,
        which gives a cell decomposition formed from tetrahedra,
        three-sided footballs, 4-sided footballs and/or triangular purses;
        \item performing a sequence of atomic operations, each of
        which either flattens a triangular pillow to a triangle,
        flattens a bigonal pillow to a bigon, or
        flattens a bigon to an edge.
    \end{enumerate}

    After step~(1) (cutting along $S$ and collapsing the boundaries),
    we obtain a cell decomposition with two components:
    one represents the non-compact manifold $\mfdi$, and one represents
    the non-compact product $S \times (0,1)$.
    These are ``ideal cell decompositions'' in the same sense as an
    ideal triangulation---if we remove the vertices whose links are
    non-spheres, then the underlying topological spaces are
    homeomorphic to $\mfdi$ and $S \times (0,1)$ respectively.

    At this stage we throw away the $S \times (0,1)$ component, and
    focus solely on the ideal cell decomposition of $\mfdi$.

    What remains is to inductively show that, if we apply any individual
    atomic operation of step~(2) to an ideal cell decomposition of $\mfdi$,
    we obtain another ideal cell decomposition of $\mfdi$ (possibly
    after throwing away more unwanted components).
    We consider each type of operation in turn; the reader may wish to
    refer back to Figure~\ref{fig-atomic} for illustrations of these
    operations.
    \begin{itemize}
        \item \emph{Flattening a triangular pillow:}

        If the two triangular faces of
        the pillow are not identified in the cell decomposition then
        this operation does not change the topology.

        If the two faces are identified then the pillow
        forms an entire connected component, and therefore
        represents the entire ideal cell decomposition of $\mfdi$.
        Here we obtain a contradiction: by enumerating
        all six ways in which the upper face can be identified to the lower,
        we see that the underlying manifold must instead be $S^3$ or $L(3,1)$,
        or else the cell decomposition must have invalid edges.

        \item \emph{Flattening a bigonal pillow:}

        As before, if the two bigon faces of
        the pillow are not identified in the cell decomposition then
        this operation does not change the topology.

        If the two bigon faces are identified then again the pillow
        must form the entire ideal cell decomposition of $\mfdi$,
        and again we obtain a contradiction.
        There are four ways in which
        the upper face can be identified to the lower,
        and these yield $S^3$, $\R P^3$, $\R P^2 \times (0,1)$, or a
        cell decomposition with invalid edges.

        \item \emph{Flattening a bigon:}

        Once again, if the two edges of the bigon are not identified in
        the cell decomposition then this operation does not change the
        topology.  If the edges are identified, however, then more
        delicate arguments are required.

        \begin{figure}[tb]
        \centering
        \subfigure[Non-orientable]{\label{fig-bigon-mb}%
            \quad\includegraphics[scale=0.7]{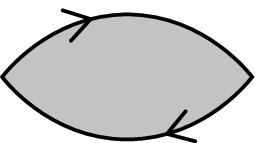}\quad}
        \qquad\qquad
        \subfigure[Orientable]{\label{fig-bigon-annulus}%
            \quad\includegraphics[scale=0.7]{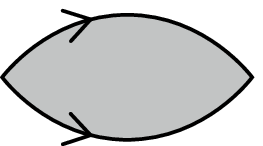}\quad}
        \caption{Identifying the two edges of a bigon}
        \end{figure}

        Suppose the edges are identified so the bigon becomes a
        non-orientable embedded surface, as in Figure~\ref{fig-bigon-mb}.
        In this case, both vertices of the bigon must be
        identified as a single vertex.  If this is an ideal vertex
        of the cell decomposition then by Observation~\ref{obs-am} the bigon
        represents a boundary-parallel {\mb} in $\mfdb$;
        we postpone this case for the moment.
        If this vertex is internal to $\mfdi$ then the bigon becomes an
        embedded projective plane, which is impossible.

        Otherwise the edges must be identified to give an orientable
        embedded surface, as in Figure~\ref{fig-bigon-annulus}.
        Here there are several cases to consider:
        \begin{itemize}
            \item Both vertices of the bigon cannot be internal.
            This is because each bigon was formed in step~(1) above
            by collapsing a copy of the surface $S$ on the boundary
            to a point, and so each bigon meets the
            ideal vertex that represents the
            corresponding boundary component of $\mfdb$.

            \begin{figure}[tb]
            \centering
            \includegraphics[scale=0.9]{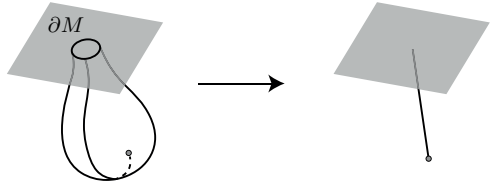}
            \caption{Flattening a bigon that represents a disc in $\mfdb$}
            \label{fig-bigon-flattendisc}
            \end{figure}

            \item If one vertex of the bigon is internal and one is ideal,
            then the bigon represents a properly embedded disc in $\mfdb$.
            Since $\mfdb$ has no essential compression discs or
            essential spheres, the boundary of this disc must be trivial in
            $\bdry \mfdb$, and so the disc and a portion of $\bdry \mfdb$
            must together bound a ball.  Flattening the bigon has
            the topological effect of collapsing this ball to an edge
            as illustrated in Figure~\ref{fig-bigon-flattendisc},
            and the result is a new ideal cell decomposition of $\mfdi$.
            The portion of the original cell decomposition that was
            inside the ball splits off into a separate component
            of the new cell decomposition, and we simply throw this
            component away.

            \begin{figure}[tb]
            \centering
            \includegraphics[scale=0.9]{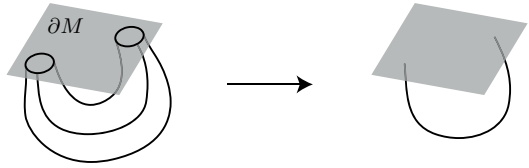}
            \caption{Flattening a bigon that represents an annulus in $\mfdb$}
            \label{fig-bigon-flattenannulus}
            \end{figure}

            \item If both vertices of the bigon are ideal,
            then the bigon represents a properly embedded annulus in $\mfdb$.
            By Observation~\ref{obs-am}, there are two possibilities.
            The annulus might be boundary-parallel in $\mfdb$; we
            postpone this case for the moment.
            Otherwise the annulus is two-sided with trivial boundary
            curves in $\bdry \mfdb$, and again flattening the bigon
            has the effect of collapsing a ball to an edge
            as shown in Figure~\ref{fig-bigon-flattenannulus}.
            As before, we throw away the interior of the ball,
            and what remains is a new ideal cell decomposition of $\mfdi$.
            Note that this argument holds even if the
            two ideal vertices of the bigon are identified.
        \end{itemize}
    \end{itemize}

    The only cases not handled above are those in which a bigon of the
    cell decomposition represents a boundary-parallel annulus or {\mb}
    in $\mfdb$.  To finish we show that, because our original normal
    surface $S$ was \emph{outermost}, such cases can never occur.

    \begin{figure}[tb]
        \centering
        \subfigure[Bigons in $C_0$]{\label{fig-chain-bigons}%
            \quad\includegraphics[scale=0.595]{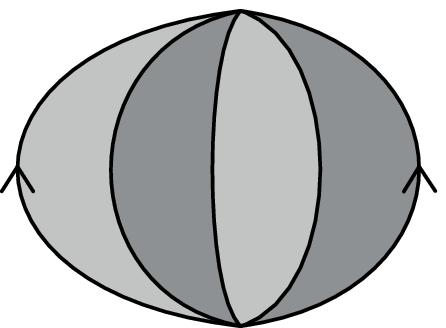}\quad}
        \subfigure[Trapezoids in $\tri$]{\label{fig-chain-trap}%
            \includegraphics[scale=0.85]{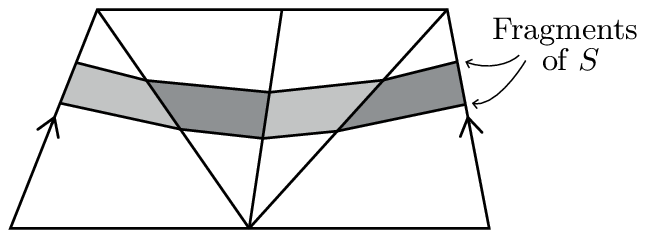}}
        \subfigure[Expanding $S$ to $S''$]{\label{fig-chain-expand}%
            \qquad\includegraphics[scale=0.85]{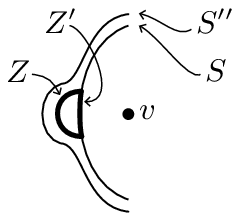}\qquad}
        \caption{Resolving boundary-parallel annuli and {\mb}s}
    \end{figure}

    Let $C_0$ denote the ideal cell decomposition obtained immediately
    after cutting $\tri$ along $S$ and collapsing the boundary
    components (i.e., before any atomic operations are performed).
    Suppose that, at some stage of the crushing process,
    we have a cell decomposition $C_i$ in which some bigon
    represents a boundary-parallel annulus or {\mb}.
    It follows that in $C_0$ there is a \emph{chain} of one or more bigons
    representing this boundary-parallel annulus or {\mb}, as illustrated
    in Figure~\ref{fig-chain-bigons} (so all but one of these bigons
    were flattened between stages $C_0$ and $C_i$).  Choose the smallest such
    chain of bigons, so that the corresponding annulus or {\mb} is
    properly embedded in $\mfdb$.

    In the original triangulation $\tri$, this chain of bigons
    corresponds to a chain of \emph{trapezoids}, as shown in
    Figure~\ref{fig-chain-trap}: each trapezoid lives within a face of
    $\tri$, and is bounded by two edge fragments of $\tri$ and two
    fragments of the normal surface $S$.  The union of these trapezoids
    forms an embedded annulus or {\mb} $Z \subset \tri$ whose boundary
    $\bdry Z$ runs along $S$, and which is parallel into some
    embedded annulus or {\mb} $Z' \subset S$.

    Let $S' = (S \backslash Z') \cup Z$; that is,
    we replace the annulus or {\mb} from within $S$ with the
    parallel union of trapezoids.  Note that $S'$ is isotopic to $S$.
    Let $S''$ be an embedded surface parallel to $S'$ but slightly
    further away from the ideal vertex $v$, as illustrated in
    Figure~\ref{fig-chain-expand}.
    Then $S''$ is also isotopic to $S$,
    disjoint from $S$, and disjoint from the union of trapezoids $Z$.
    Moreover, both $S$ and $S'$ lie inside the collar between $S''$
    and a small neighbourhood of $v$.

    We now normalise this surface $S''$; let $N$ denote the
    resulting normal surface in $\tri$.
    Since $S''$ is boundary-parallel and $\mfdb$ has no essential
    compression discs, the normalised surface $N$ is again
    isotopic to $S$.  Since $S$ and $Z$ together form a
    barrier to normalisation (Section~\ref{s-prelim-barrier}),
    it follows that $N$ is disjoint from both $S$ and $Z$,
    and that $S$ still lies inside the collar between $N$ and a small
    neighbourhood of $v$ (i.e., the normalisation process does not
    ``cross'' through $S$).
    Finally, because $N$ is disjoint from the union of trapezoids $Z$,
    it must be a \emph{different} normal surface; i.e., $N$ is not
    \emph{normally} isotopic to the original surface $S$.
    This contradicts our assumption that $S$ was outermost, and the
    proof is complete.
\end{proof}

\begin{corollary} \label{c-noparallel}
    Let $\mfdb$ be a compact 3-manifold with boundary whose interior
    $\mfdi$ is a {\hypmfd}, and let $\tri$ be a minimal ideal
    triangulation of $\mfdi$.
    Then the only boundary-parallel normal surfaces in $\tri$ are
    vertex linking (i.e., they consist only of triangles).
\end{corollary}

\begin{proof}
    Suppose $\tri$ has some non-vertex-linking, boundary-parallel
    normal surface $S$.
    Without loss of generality we may assume that $S$ is
    connected, i.e., parallel onto a single boundary component of $M$;
    let $v$ be the corresponding ideal vertex of $\tri$.

    There must exist a normal surface $S'$ that is isotopic to $S$
    and outermost with respect to $v$, since
    by a standard Kneser-type finiteness argument \cite{kneser29-normal}
    there cannot be infinitely many disjoint normal surfaces
    in $\tri$ with no pair normally isotopic.
    Moreover, this outermost surface $S'$ must also be non-vertex-linking,
    since the original non-vertex-linking surface $S$ cannot be placed inside
    the collar between the vertex link of $v$ and a small neighbourhood of $v$.

    It follows from Lemma~\ref{l-crushsame} that crushing
    $S'$ gives a new ideal triangulation $\tri'$ of $\mfdi$.
    Because $S'$ is non-vertex linking, $\tri'$ must contain
    strictly fewer tetrahedra than $\tri$, contradicting the
    minimality of $\tri$.
\end{proof}

We can now use the results above to prove some combinatorial properties
of minimal triangulations of hyperbolic manifolds.

\begin{theorem} \label{t-nointernal}
    Let $\tri$ be a minimal ideal triangulation of a {\hypmfd}.
    Then $\tri$ contains no internal vertices.
\end{theorem}

\begin{proof}
    Let $\mfdb$ denote the corresponding compact 3-manifold with boundary.
    If $\tri$ has an internal vertex, then $\tri$ has an edge $e$ joining
    some internal vertex $u$ to some ideal vertex $v$.
    Let $S$ be an embedded surface that bounds
    a small regular neighbourhood of $e$.
    Then $S$ is boundary-parallel onto $v$.

    We now normalise $S$; let $N$ denote the corresponding normal
    surface in $\tri$.  Since $S$ is boundary-parallel and $\mfdb$ has no
    essential compression discs, $N$ must be isotopic to $S$ and therefore
    also boundary-parallel onto $v$.  Since the edge $e$ acts as a barrier to
    normalisation (Section~\ref{s-prelim-barrier}), $S$ cannot normalise to
    the vertex link of $v$, and so $N$ is a non-vertex-linking
    boundary-parallel normal surface in contradiction to
    Corollary~\ref{c-noparallel}.
\end{proof}

\begin{theorem} \label{t-nospheres}
    Let $\tri$ be a minimal ideal triangulation of a {\hypmfd}.
    Then $\tri$ contains no normal 2-spheres.
\end{theorem}

\begin{proof}
    This is essentially a variant of the proof of Theorem~\ref{t-nointernal}.
    Again let $\mfdb$ denote the corresponding compact 3-manifold with boundary.

    If $S$ is a normal 2-sphere in $\tri$, then $S$ must bound a ball $B$.
    Moreover, since every vertex of $\tri$ is ideal
    (Theorem~\ref{t-nointernal}), every vertex of $\tri$
    lies outside this ball.

    Let $e$ be any edge of $\tri$ that meets $S$, let $v$ be the ideal
    vertex at some end of $e$, and let $\alpha \subset e$
    denote the segment of $e$ that runs from $v$ along
    to the first point at which $e$ meets $S$.

    \begin{figure}[tb]
        \centering
        \includegraphics[scale=0.9]{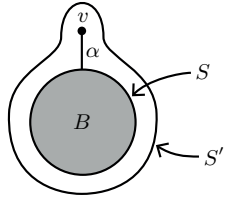}
        \caption{Extending a normal 2-sphere to a boundary-parallel surface}
        \label{fig-nospheres}
    \end{figure}

    Let $S'$ denote the boundary of a small regular neighbourhood of
    $B \cup \alpha$ in $\tri$, as illustrated in Figure~\ref{fig-nospheres}.
    Then $S'$ is boundary-parallel onto $v$.

    As before, we normalise $S'$; let $N$ denote the corresponding normal
    surface in $\tri$.  Again the normalisation $N$ must also be
    boundary-parallel onto $v$.  This time the normal surface $S$ and
    the arc $\alpha$ together act as a barrier to normalisation,
    and so once again $S'$ cannot normalise to the vertex link of $v$.
    Therefore $N$ is a non-vertex-linking
    boundary-parallel normal surface, in contradiction to
    Corollary~\ref{c-noparallel}.
\end{proof}

\begin{theorem} \label{t-notorikb}
    Let $\tri$ be a minimal ideal triangulation of a {\hypmfd} $\mfdi$.
    Then any non-vertex-linking normal torus or Klein bottle in $\tri$
    must bound a solid torus ($B^2 \times S^1$)
    or solid Klein bottle ($B^2 \twisted S^1$) respectively.
\end{theorem}

\begin{proof}
    Let $S$ be a non-vertex-linking normal torus or Klein bottle in $\tri$.
    We take cases according to whether $S$ is a torus or Klein bottle,
    and whether it is two-sided or one-sided.

    Suppose $S$ is a two-sided torus.
    By Corollary~\ref{c-noparallel}, $S$ is not boundary-parallel;
    since $\mfdi$ contains no essential tori, it follows that $S$ is not
    $\pi_1$-injective (see Section~\ref{s-prelim-hyp}).
    Therefore $S$ has a compression disc $D$;
    that is, an embedded disc $D \subset \mfdi$ for which
    $D \cap S = \bdry D$ and $\bdry D$ is a non-trivial curve in $S$.

    \begin{figure}[tb]
        \centering
        \includegraphics[scale=0.9]{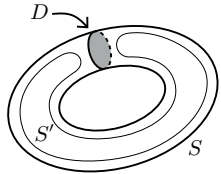}
        \caption{A compression disc for the torus $S$ and the sphere $S'$}
        \label{fig-torusdisc}
    \end{figure}

    Let $S'$ denote the sphere on the boundary of a regular
    neighbourhood of $S \cup D$,
    as illustrated in Figure~\ref{fig-torusdisc}.  This sphere $S'$ must
    bound a ball in $\mfdi$.  If $S'$ bounds a ball away from $S \cup D$,
    then the normal torus $S$ bounds a solid torus.  Otherwise we can
    normalise $S'$ to a \emph{normal} sphere in $\mfdi$, since
    the normal torus $S$ acts a barrier to normalisation,
    and since the sphere $S'$ is essential in $\mfdi \backslash S$.
    This contradicts Theorem~\ref{t-nospheres}.

    Next, suppose that $S$ is a two-sided Klein bottle.
    Again Corollary~\ref{c-noparallel} shows
    that $S$ cannot be boundary-parallel.
    It follows that, in the orientable double cover of $\mfdi$,
    $S$ lifts to a two-sided torus that cannot be $\pi_1$-injective, and
    so $S$ is not $\pi_1$-injective in $\mfdi$.  We therefore obtain a
    compression disc $D \subset \mfdi$ for the two-sided Klein bottle $S$
    as before.

    As in the previous case, let $S'$ denote the sphere on
    the boundary of a regular neighbourhood of $S \cup D$.
    Since a ball cannot contain an embedded Klein bottle,
    $S'$ must bound a ball away from $S \cup D$, and therefore the normal Klein
    bottle $S$ bounds a solid Klein bottle.

    Finally, suppose that $S$ is one-sided.  Let $2S$ denote the double of $S$.
    Then $2S$ is a non-vertex-linking two-sided normal torus or Klein bottle,
    and by the arguments above $2S$ bounds a solid torus or Klein bottle
    in $\mfdi$.  This is impossible, since $2S$ has the ideal
    vertices of $\tri$ on one side and a twisted $I$-bundle over the torus or
    Klein bottle on the other.
\end{proof}

\begin{theorem} \label{t-lowdeg}
    Let $\tri$ be a minimal ideal triangulation of a {\hypmfd}.
    Then $\tri$ has no edges of degree~1 or~2,
    and $\tri$ has no edges of degree~3 that are contained in
    three distinct tetrahedra.
\end{theorem}

\begin{proof}
    We recall from Theorem~\ref{t-nointernal} that every vertex of
    $\tri$ is ideal.  Let $n$ denote the number of tetrahedra in $\tri$.

    \begin{figure}[tb]
        \centering
        \subfigure[Near a degree~1 edge]{\label{fig-deg1}%
            \quad\includegraphics[scale=1.0]{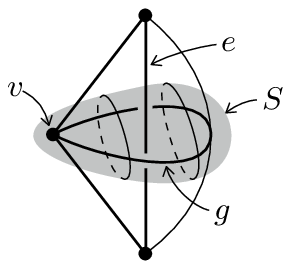}\quad}
        \qquad\qquad
        \subfigure[Near a degree~2 edge]{\label{fig-deg2}%
            \includegraphics[scale=1.0]{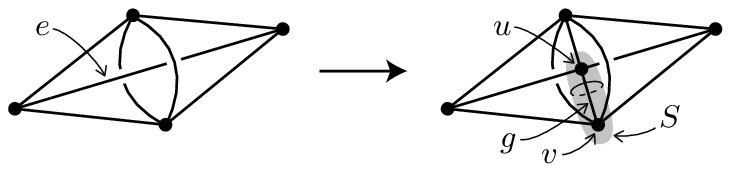}}
        \caption{Building boundary-parallel surfaces from low-degree edges}
    \end{figure}

    Suppose $\tri$ has an edge $e$ of degree~1.  Then there is some
    tetrahedron $\Delta$ with two faces folded together around $e$,
    as illustrated in Figure~\ref{fig-deg1}.  Let $g$ denote the
    edge of $\Delta$ that encircles $e$, and let $v$ denote the ideal vertex
    that meets $g$.
    Form an embedded surface $S$ that is boundary-parallel onto $v$
    and that encloses the edge $g$, as shown in the illustration.
    Since $g$ acts a barrier to normalisation, $S$ must normalise to a
    surface that is boundary-parallel onto $v$ but not the vertex link of $v$,
    contradicting Corollary~\ref{c-noparallel}.

    Suppose instead that $\tri$ has an edge $e$ of degree~2.
    Then there are two tetrahedra
    joined together along two faces on either side of $e$,
    as illustrated in the left portion of Figure~\ref{fig-deg2}.
    We subdivide $e$ with a new internal vertex $u$,
    and subdivide these two original tetrahedra into four
    as shown in the right portion of Figure~\ref{fig-deg2},
    so the resulting triangulation contains $n+2$ tetrahedra in total.
    Let $g$ denote one of the edges contained in all four new tetrahedra,
    as shown in the diagram.
    Let $v$ denote the ideal vertex that meets $g$, and let
    $S$ denote the boundary of a small regular neighbourhood of $g$.
    We observe that: (i)~$S$ is boundary-parallel onto the ideal vertex $v$;
    (ii)~$S$ is in fact a \emph{normal} surface; and
    (iii)~$S$ contains a quadrilateral in each of the four new tetrahedra.

    Let $S'$ be an \emph{outermost} normal surface that is
    boundary-parallel onto $v$ and that contains $S$ in the collar
    between $S'$ and $v$.
    As in the proof of Corollary~\ref{c-noparallel}, a standard
    Kneser-type finiteness argument \cite{kneser29-normal} shows that
    such a surface $S'$ exists (if $S$ is already outermost then
    $S'$ will just be normally isotopic to $S$).
    It follows from Lemma~\ref{l-crushsame} that crushing $S$ yields a
    new ideal triangulation $\tri'$ of the original manifold.
    Moreover, by observation~(iii) above, all four of the new tetrahedra
    will be destroyed by the crushing process, and so $\tri'$ will
    contain at most $n-2$ tetrahedra, contradicting the minimality of
    $\tri$.

    Finally,
    suppose $\tri$ has an edge of degree~3 contained in three distinct
    tetrahedra.  Here we can perform a 3-2 Pachner move to reduce the number of
    tetrahedra, again contradicting the minimality of $\tri$.
\end{proof}

%
%

\section{Generating candidate triangulations} \label{s-gen}

In this section we describe the algorithmic process of generating all
candidate ideal
3-manifold triangulations with $n$ tetrahedra, for each $n=1,\ldots,9$.
Here we use our combinatorial constraints on vertex links and low-degree
edges (Theorems~\ref{t-nointernal} and~\ref{t-lowdeg}) as
an integral part of the enumeration algorithm.

The algorithm extends earlier census algorithms in a way that is
mathematically straightforward (though more intricate to code),
and so we give only a brief outline of the process here.  The
full source code can be viewed in {\regina}'s online code repository
\cite{regina}, and will be included in the coming release of
\regina~4.96.

\begin{table}[tb]
\caption{Generation of candidate triangulations}
\label{tab-gen}
\begin{tabular}{c|r|r}
Tetrahedra & \multicolumn{2}{c}{Triangulations} \\
\cline{2-3}
& \multicolumn{1}{c|}{Results of}
& \multicolumn{1}{c}{Including low-} \\
& \multicolumn{1}{c|}{Theorem~\ref{t-enumerate}}
& \multicolumn{1}{c}{degree edges} \\
\hline
1 &           1 &                1 \\
2 &           7 &               18 \\
3 &          31 &              246 \\
4 &         224 &           3\,503 \\
5 &      1\,075 &          51\,652 \\
6 &      6\,348 &         810\,473 \\
7 &     35\,312 &     13\,090\,995 \\
8 &    218\,476 &    216\,484\,558 \\
9 & 1\,313\,052 & 3\,625\,523\,250 \\
\hline
Total & 1\,574\,526 & 3\,855\,964\,696
\end{tabular}
\end{table}

In summary, the results are:
\begin{theorem} \label{t-enumerate}
    Consider ideal 3-manifold triangulations in which
    every vertex link is a torus or Klein bottle,
    there are no edges of degree~1 or~2, and
    there are no edges of degree~3 that belong to three distinct
    tetrahedra.
    Up to relabelling, there are precisely
    $1\,574\,526$
    such triangulations with $n \leq 9$ tetrahedra.
    Moreover, every minimal ideal triangulation of a {\hypmfd}
    with \mbox{$n \leq 9$} tetrahedra belongs to this set.
\end{theorem}

\begin{proof}
    We obtain this count of
    $1\,574\,526$ triangulations by explicitly generating them,
    as outlined below in Algorithm~\ref{a-enum}.
    The middle column of Table~\ref{tab-gen}
    breaks this figure down by number of tetrahedra,
    and Section~\ref{s-census} lists the computational running times.
    It is immediate from
    Theorems~\ref{t-nointernal} and~\ref{t-lowdeg} that
    every minimal triangulation of a {\hypmfd} belongs to this set.
\end{proof}

In brief, the enumeration algorithm operates as follows.

\begin{algorithm} \label{a-enum}
    To enumerate all generalised triangulations with $n$ tetrahedra:
    \begin{enumerate}
        \item \label{en-graph}
        Enumerate all connected 4-valent multigraphs on $n$ nodes
        (here loops and multiple edges are allowed).
        These will become the dual 1-skeleta, or face pairing graphs, of
        our triangulations---their nodes represent tetrahedra, and their
        arcs represent identifications between tetrahedron faces.

        \item \label{en-glue}
        For each such multigraph $\Gamma$, recursively try all
        possible ways of identifying the corresponding pairs of
        tetrahedron faces.  Note that, for each arc of $\Gamma$,
        there are six ways in which the corresponding pair of faces
        could be identified (corresponding to the six symmetries
        of the triangle).
    \end{enumerate}

    To ensure that each triangulation appears only once up to
    relabelling, we only keep those triangulations that are
    ``canonical''.  Essentially this means that,
    when the pairwise face identifications are
    expressed as a sequence of integers, this sequence is
    lexicographically minimal under all possible relabellings.

    To ensure that we only obtain ideal 3-manifold triangulations in which
    every vertex link is a torus or Klein bottle,
    that there are no edges of degree~1 or~2, and
    that there are no edges of degree~3 that belong to three distinct
    tetrahedra, we adapt step~(\ref{en-glue}) as follows.
    Consider each branch of the recursion, where we have a
    ``partially-constructed'' triangulation where only some of the
    face identifications have been selected.
    We prune this branch and backtrack immediately if we can show that,
    no matter how we complete our triangulation, we must obtain either:
    \begin{itemize}
        \item a non-canonical triangulation;
        \item an invalid edge;
        \item an edge of degree $\leq 2$, or an edge of degree 3 that
        meets three distinct tetrahedra;
        \item a vertex link with non-zero Euler characteristic.
    \end{itemize}
\end{algorithm}

These pruning tests are run extremely often (the recursion tree has
$6^{2n}$ branches for each multigraph $\Gamma$),
and so it is imperative that they be extremely fast---a na{\"i}ve
implementation could ultimately slow the enumeration down even whilst
reducing the size of the underlying search tree.
We address this as follows:
\begin{itemize}
    \item We construct the automorphism group of the multigraph $\Gamma$
    (i.e., the group of relabellings that leave $\Gamma$ unchanged),
    and use this to detect situations in which any completion of our
    partial triangulation must be non-canonical.
    See \cite{burton03-thesis} for details.
    Automorphisms have a long history of use in related combinatorial
    algorithms from graph theory \cite{mckay81-practical}.

    \item To detect invalid edges and low-degree edges,
    we track equivalence classes of
    tetrahedron edges (according to how they are identified within the
    partial triangulation), along with associated orientation
    information.  We maintain these equivalence classes using a
    modification of the \emph{union-find} data structure that
    allows not only for fast
    merging of equivalence classes (which union-find excels at) but also
    fast backtracking (which we need for our recursion).
    Details appear in \cite{burton07-nor10}.

    \item To detect vertex links with non-zero Euler characteristic,
    we likewise track equivalence classes of tetrahedron vertices,
    along with genus-related information.
    Specifically, each equivalence class represents a single vertex in the
    partial triangulation, and we require at all times that the
    link of such a vertex must be
    (i)~a sphere with one or more punctures; (ii)~a projective plane
    with one or more punctures; or (iii)~a torus or Klein bottle with
    zero or more punctures.

    The paper \cite{burton11-genus} describes fast data structures based
    on union-find and skip lists for maintaining equivalence classes of
    vertices and ensuring that every vertex link is a sphere with
    zero or more punctures (a condition tailored for the setting of
    closed 3-manifolds, not ideal triangulations).
    It is straightforward to adapt this to our setting:
    for each equivalence class we now maintain the pair $(\chi, p)$,
    where the corresponding vertex link is a closed surface of Euler
    characteristic $\chi$ with $p$ punctures.
    Conditions (i), (ii) and (iii) simply translate to $\chi \geq 0$,
    with $p > 0$ if $\chi > 0$.
    The same fast data structures described in \cite{burton11-genus}
    can be used to update the pairs $(\chi,p)$ as we merge
    equivalence classes and as we backtrack.
\end{itemize}

We note that the overall framework of enumerating graphs and then
gluings (i.e., the separation of steps~(1) and~(2) above)
is common to most 3-manifold census papers in the literature.
See \cite{hildebrand89-cuspedcensusold,matveev98-or6}
for some early examples.
With regard to pruning,
Callahan et~al.\ also maintain genus-related
data for partially-constructed vertex links, though they give no
further information on their underlying data structures
\cite{callahan99-cuspedcensus}.

The final column of Table~\ref{tab-gen} highlights the practical importance of
our combinatorial results from Section~\ref{s-min}.
If we remove the constraints on low-degree edges, there are over
$3.8$~billion ideal 3-manifold triangulations with $n \leq 9$ tetrahedra
in which every vertex link is a torus or Klein bottle---this
is several orders of magnitude more triangulations to process.
Moreover, by using the low-degree edge constraint to prune the search tree
(as described above),
we cut the enumeration time from over 1.2 CPU~years to roughly 1.4 CPU~months.
This highlights why such constraints should be embedded directly into the
enumeration algorithm where practical, instead of using them to
discard unwanted triangulations \emph{after} they have been built.

%
%

\section{Certifying non-hyperbolicity} \label{s-process}

We now describe the suite of algorithmic tests with which we
prove Theorem~\ref{t-falseneg} and Theorem~\ref{t-complete9}.
%
Our overall strategy is, for each of the $1\,574\,526$
distinct triangulations obtained in
Theorem~\ref{t-enumerate}, to run a suite of tests
that attempt to quickly certify one of the following:
(i)~that the triangulation is non-minimal and/or
the underlying manifold is non-hyperbolic; or
(ii)~that the underlying manifold is hyperbolic and is contained in the census.

The bulk of the computational work, and the main focus of this
section, is on case~(i).
It is important to note that \emph{we do not need to distinguish
between non-minimality and non-hyperbolicity}, since either allows us to
discard a triangulation---even if a non-minimal triangulation represents
a hyperbolic manifold, we will have processed this same manifold before.
This observation simplifies some of our tests significantly.

Our suite of tests does not guarantee to certify one of the two outcomes
above for every triangulation.  However:
\begin{itemize}
    \item These tests are \emph{rigorous}: any results they
    do produce are based on exact computations, and give
    conclusive certificates of (i) or (ii) accordingly.

    \item These tests are \emph{fast}:
    we are able to process all $1\,574\,526$ triangulations
    in under $16$~hours of CPU time, or $<0.04$\,seconds
    per triangulation on average.

    \item These tests are \emph{effective}: amongst all
    $1\,574\,526$ triangulations, there are just $396$ for which they do not
    conclusively prove (i) or (ii) above.
\end{itemize}

In the remainder of this section
we describe this suite of tests in detail, with proofs of
correctness where necessary.
Later, in Section~\ref{s-census}, we describe the way in which we
combine these tests, give detailed running times,
and wrap up the proofs of our main results
(Theorem~\ref{t-falseneg}, Corollary~\ref{c-complete8} and
Theorem~\ref{t-complete9}).
There we also explain how we handle the 396 remaining triangulations
that these tests do not resolve.

We emphasise that theoretical algorithms are already known for certifying
non-hyperbolicity, and indeed our tests draw upon these earlier results.
The value of the tests in this paper is that they are not only correct but also
fast and effective,
and therefore well-suited for bulk processing with millions of triangulations
as outlined above.


\subsection{Local moves}

Our first tests are the most elementary: we try to \emph{simplify} the
triangu\-lation---that is, retriangulate the underlying
manifold using fewer tetrahedra---in order to prove non-minimality.
There is a trade-off here between speed and power, and so we use two
different approaches:
a polynomial-time greedy test,
and a super-exponential-time exhaustive test.
Both are based on well-known techniques.

The paper \cite{burton13-regina} describes a greedy
algorithm that attempts to simplify a given 3-manifold triangulation.
It is based on local modifications to the triangulation,
including the standard Pachner moves (or bistellar flips)
plus a variety of more complex moves.
Most of these moves have been known for a long time, and have been
used by many authors in a variety of settings.

The algorithm is greedy in the sense that it will never
increase the number of tetrahedra at any step (i.e., it does not
attempt to climb out of ``wells'' around local minima).
As a result, it can prove triangulations to be non-minimal
(and it is extremely effective at this),
but it can never prove a triangulation to be minimal.

\begin{test}[Greedy non-minimality test] \label{test-greedy}
    Let $\tri$ be an ideal 3-manifold triangulation with $n$ tetrahedra.
    Run the greedy simplification algorithm of \cite{burton13-regina}
    over $\tri$.  If this results in a triangulation with fewer
    than $n$ tetrahedra then $\tri$ is non-minimal.
\end{test}

This test is one of the fastest in our suite, with a
small polynomial running time of $O(n^4 \log n)$.
See \cite[Theorem~2.6]{burton13-regina} for a detailed
time complexity analysis.

For triangulations where greedy methods fail, one can take a more
exhaustive approach.  The paper \cite{burton11-pachner} describes an
algorithm to enumerate \emph{all} triangulations that can be reached
from a given triangulation $\tri$ by performing \emph{any} sequence
of 2-3 and 3-2 Pachner moves without ever exceeding a given upper limit
on the number of tetrahedra.  The algorithm is based on a breadth-first
search, and uses \emph{isomorphism signatures} to avoid revisiting the
same triangulations; these are polynomial-time computable hashes
that uniquely identify a triangulation up to relabelling.

\begin{test}[Exhaustive non-minimality test] \label{test-exhaustive}
    Let $\tri$ be an ideal 3-manifold triangulation with $n$ tetrahedra,
    and let $h \in \N$.
    Compute all triangulations that can be obtained from $\tri$
    by performing 2-3 and 3-2 Pachner moves without ever exceeding
    $n+h$ tetrahedra.  If any such triangulation has fewer than $n$
    tetrahedra then $\tri$ is non-minimal.
\end{test}

This exhaustive test is much slower,
with a super-exponential running time of $\exp(O(n \log n))$
for fixed $h$.
This is still not
a severe problem (in our setting, it just means that the running time is
measured in seconds as opposed to microseconds).  However, it does mean
that we cannot run this exhaustive test over all $1\,574\,526$
triangulations; instead we must reserve it for difficult cases
where simpler tests have failed.

Like the greedy test,
this exhaustive test can identify non-minimal
triangulations but (for any reasonable $h$) cannot prove minimality.
However, it is extremely effective at simplifying triangulations
that the greedy algorithm cannot, even for $h$ very small.
By default we run this test with $h=2$, which in the related
setting of closed manifolds is a natural threshold\footnote{%
    This relates to the fact that every closed 3-manifold
    triangulation has an edge of degree $\leq 5$.},
and which is enough to simplify
\emph{all} $\sim 31$~million distinct triangulations of the 3-sphere
for $n \leq 9$ \cite{burton11-pachner}.


\subsection{Normal surfaces}

As seen in Section~\ref{s-gen},
some of our combinatorial results from Section~\ref{s-min}
can be embedded directly into the enumeration algorithm
(e.g., no low-degree edges and no internal vertices).
Others cannot (e.g., the constraints involving normal surfaces),
and so we use these here instead as tests for non-minimality.

The following tests follow directly from Theorems~\ref{t-nospheres}
and~\ref{t-notorikb}, plus the fact that a {\hypmfd} cannot contain
an embedded projective plane.

\begin{test}[Spheres and projective planes] \label{test-eff}
    Let $\tri$ be an ideal 3-manifold triangulation.
    Enumerate all standard vertex normal surfaces in $\tri$.
    If any of these surfaces has positive Euler characteristic,
    then $\tri$ is either non-minimal or does not represent a
    {\hypmfd}.
\end{test}

\begin{test}[Tori and Klein bottles] \label{test-torikb}
    Let $\tri$ be an ideal 3-manifold triangulation.
    Enumerate all standard vertex normal surfaces in $\tri$,
    and test whether any such surface is a torus or Klein bottle
    $S$ for which:
    \begin{itemize}
        \item $S$ is one-sided; or
        \item $S$ is two-sided, and if we cut $\tri$ open along $S$
        then no component of the resulting triangulation is a
        solid torus or solid Klein bottle.
    \end{itemize}
    If such a surface is found,
    then $\tri$ is either non-minimal or does not represent a {\hypmfd}.
\end{test}

Test~\ref{test-torikb} requires us to algorithmically recognise
the solid torus and solid Klein bottle.  For the solid torus we use a
standard crushing-based algorithm, which we describe in more detail
later (Algorithm~\ref{a-solidtorus}).  For the solid Klein bottle we run
solid torus recognition over the orientable double cover.

Regarding running times:
Test~\ref{test-eff} runs in time $O(7^n \times \mathrm{poly}(n))$---the
bottleneck here is simply enumerating all standard vertex normal surfaces.
See \cite{burton13-tree} for the full time complexity analysis.
Test~\ref{test-torikb} may require doubly-exponential time in theory,
since cutting along $S$ could introduce exponentially many tetrahedra.
In practice however, vertex normal surfaces typically have very few
normal discs \cite{burton13-bounds} and so the resulting triangulations
remain manageably small.

We emphasise again
how our tests become simpler because we do not need to
distinguish between non-minimality and non-hyperbolicity.
It is enough just to
find normal spheres, tori or Klein bottles as described above---there is
no need to test whether these surfaces are essential (which, for tori
and Klein bottles in particular, would be a more expensive process).

The tests above only examine \emph{vertex} normal surfaces,
not arbitrary normal surfaces.
This is to make the tests fast enough for our bulk-processing requirements.
For Test~\ref{test-eff}, this restriction
does not reduce the power of the test at all---a simple linearity argument
shows that $\tri$ contains a standard vertex normal surface of positive Euler
characteristic if and only if it contains \emph{any} normal surface of
positive Euler characteristic.
For Test~\ref{test-torikb} we may lose some opportunities to prove
non-minimality or non-hyperbolicity, but as we see later in
Section~\ref{s-census}, this does not hurt us significantly in practice.

We use \emph{standard} vertex normal surfaces instead of quadrilateral
vertex normal surfaces because this ensures that the vertex surfaces
are closed.  In ideal triangulations, quadrilateral vertex surfaces
typically include non-compact \emph{spun-normal surfaces}
\cite{tillmann08-finite}, which we wish to avoid here.


\subsection{Seifert fibred spaces}

Some Seifert fibred spaces contain no essential tori or Klein bottles
at all---examples include Seifert fibred spaces over the disc, annulus
or pair of pants with at most two, one and zero exceptional fibres
respectively.
Given an ideal triangulation of such a space that is both minimal and
sufficiently ``nice'' (i.e., with no ``unnecessary'' normal spheres,
tori or Klein bottles), all of the tests seen thus far will be inconclusive.
The following tests aim to resolve such cases.

Our focus now is purely on certifying non-hyperbolicity
(as opposed to non-minimality), and so we allow ourselves to alter the
triangulation where this is more convenient.  In particular,
we truncate ideal vertices to obtain real boundary triangles,
and for non-orientable manifolds we work in the
orientable double cover.

Our first test in this section is extremely opportunistic, but also
extremely fast and surprisingly effective.  This uses the
\emph{combinatorial recognition} routines built into {\regina}.
In essence, we examine the combinatorics of the triangulation to see if
it uses a standard combinatorial construction
for Seifert fibred spaces, and if it
does, we simply ``read off'' the Seifert fibred space parameters to identify
the underlying 3-manifold.

In brief, the standard construction
involves (i)~starting with a minimal triangulation of a punctured surface $F$;
(ii)~expanding each triangle of the surface into a 3-tetrahedron
triangular prism, giving a triangulation of $F \times I$;
(iii)~gluing the two ends of each prism together, which converts this
into $F \times S^1$; and then (iv)~attaching
layered solid tori\footnote{%
    These are simple one-vertex triangulations of the solid torus whose
    edges provide some desired set of curves on the torus boundary.
    See \cite{jaco03-0-efficiency} for details.}
to some of the remaining torus boundary components
to provide the exceptional fibres.

The reason combinatorial recognition is effective is because,
after we \emph{simplify} a triangulation of a Seifert fibred space
using the greedy algorithm from Test~\ref{test-greedy},
the result is often found to be a standard construction of this type.
See \cite{burton13-regina} for further discussion of
{\regina}'s combinatorial recognition facilities.

\begin{test}[Combinatorial recognition] \label{test-comb}
    Let $\tri$ be an ideal 3-manifold triangulation.
    Truncate the ideal vertices of $\tri$, and switch to the orientable
    double cover if $\tri$ is non-orientable.
    Simplify the resulting triangulation using the greedy algorithm of
    Test~\ref{test-greedy}.  If the result is a standard construction of
    a Seifert fibred space as outlined above, then $\tri$ does not
    represent a {\hypmfd}.
\end{test}

The test above is fast because every step
(truncation, double cover, greedy simplification and combinatorial recognition)
runs in small polynomial time \cite{burton13-regina}.

The next test is significantly more expensive, but it can succeed where
combinatorial recognition fails.  Here we explicitly search for a
normal annulus that:
\begin{itemize}
    \item for a Seifert fibred space $\mfdb$ over the disc with $\leq 2$
    exceptional fibres, splits $\mfdb$ into a pair of solid tori;
    \item for a Seifert fibred space $\mfdb$ over the annulus with $\leq 1$
    exceptional fibre, splits $\mfdb$ open into a single solid torus;
    \item for a Seifert fibred space $\mfdb$ over the pair of pants with no
    exceptional fibres, splits $\mfdb$ into a
    pair of products $\mathrm{Torus} \times I$.
\end{itemize}

As with our earlier normal surface tests, we restrict our search to
\emph{vertex} normal surfaces.  This time we work with quadrilateral
vertex surfaces,
since for manifolds with boundary triangles there are typically
far more standard vertex surfaces than quadrilateral vertex surfaces,
and the ``important'' surfaces typically appear in both sets.
As before, this restriction is primarily designed to keep the test fast---it
may cause us to lose opportunities to recognise Seifert fibred spaces,
but we see in Section~\ref{s-census} that this does not hurt us in practice.

\begin{test}[Annuli and {\mb}s] \label{test-annuli}
    Let $\tri$ be an ideal 3-manifold triangulation.
    Truncate the ideal vertices of $\tri$, and switch to the orientable
    double cover if $\tri$ is non-orientable.
    Simplify the resulting triangulation using the greedy algorithm of
    Test~\ref{test-greedy}, and denote the resulting triangulation by $\tri'$.

    Now enumerate all quadrilateral vertex normal surfaces in $\tri'$.
    If such a surface $S$ satisfies any of the following conditions,
    then the original triangulation $\tri$ does not represent a
    {\hypmfd}:
    \begin{enumerate}
        \item \label{en-test-mb}
        $S$ is a {\mb};
        \item \label{en-test-ann1}
        $S$ is an annulus, $\tri'$ has precisely one boundary component,
        and if we cut $\tri'$ open along $S$ then we obtain two solid tori;
        \item \label{en-test-ann2}
        $S$ is an annulus, $\tri'$ has precisely two boundary components,
        and if we cut $\tri'$ open along $S$ then we obtain one solid torus
        (and nothing else);
        \item \label{en-test-ann3}
        $S$ is an annulus, $\tri'$ has precisely three boundary components,
        and if we cut $\tri'$ open along $S$ then we obtain two
        triangulations of the product $\mathrm{Torus} \times I$.
    \end{enumerate}
\end{test}

\noindent
We give a proof of correctness shortly, but first
some implementation details:
\begin{itemize}
    \item Like Test~\ref{test-torikb} before,
    this test requires us to recognise the solid torus.
    This we do using the standard crushing-based algorithm
    outlined below (Algorithm~\ref{a-solidtorus}),
    and the result is guaranteed to be conclusive.
    \item This test also requires us to recognise $\mathrm{Torus} \times I$.
    This we do using combinatorial recognition, as outlined earlier in
    Test~\ref{test-comb}.
    The result might be inconclusive (i.e., we cannot tell whether a
    triangulation represents $\mathrm{Torus} \times I$ or not), but
    in such a case we simply abandon the current annulus $S$ and
    move on to the next vertex normal surface.
    Again we are making a trade-off: using combinatorial recognition to test for
    $\mathrm{Torus} \times I$ keeps the test fast but may come with an
    opportunity cost, yet again we see in Section~\ref{s-census} that
    this trade-off does not appear to hurt us in practice.
\end{itemize}

\begin{proof}[Proof for Test~\ref{test-annuli}]
    Let $\mfdi$ denote the underlying non-compact manifold for $\tri$,
    and suppose that $\mfdi$ is a {\hypmfd}.
    Let $\mfdb'$ denote the compact orientable manifold represented by $\tri'$,
    and suppose that the normal annulus or {\mb} $S \subset \mfdb'$
    satisfies one of the conditions
    (\ref{en-test-mb})--(\ref{en-test-ann3}) above.

    By Observation~\ref{obs-am}, the orientable manifold $\mfdb'$
    cannot contain a properly embedded {\mb} at all;
    that is, condition~(\ref{en-test-mb}) is impossible.
    Therefore $S$ must be an annulus.
    If $S$ is boundary-parallel then we must be in
    case~(\ref{en-test-ann1}) and $\mfdb'$ itself must also be a solid torus,
    contradicting hyperbolicity.

    Therefore $S$ is not boundary-parallel,
    and so (by Observation~\ref{obs-am} again)
    $S$ is an annulus whose boundary curves are trivial in $\bdry \mfdb'$.
    Taking cases (\ref{en-test-ann1})--(\ref{en-test-ann3}) separately:
    \begin{itemize}
        \item \emph{Case~\ref{en-test-ann1}:}
        Suppose $\mfdb'$ has one boundary component and
        $S$ separates $\mfdb'$ into a pair of solid tori.
        The only way to arrange this is for the two trivial curves of
        $\bdry S$ to sit one inside the other on $\bdry \mfdb'$,
        as shown in Figure~\ref{fig-ann1}.
        In this case $\mfdb'$ is obtained by attaching the two solid
        tori along an annulus that is trivial on one torus boundary but
        non-trivial on the other;
        this makes $\mfdb'$ a connected sum of the solid torus with one of
        $S^3$, $S^2 \times S^1$ or a lens space, any of which
        contradicts hyperbolicity.

        \item \emph{Case~\ref{en-test-ann2}:}
        Suppose $\mfdb'$ has two boundary components and
        $S$ cuts $\mfdb'$ open into a single solid torus $T$.
        This is impossible, since if $\bdry S$ consists of trivial curves on
        $\bdry \mfdb'$ then cutting $\mfdb'$ open along $S$ will always
        leave at least two boundary components.

        \begin{figure}[tb]
        \centering
        \subfigure[One boundary component]{\label{fig-ann1}%
            \qquad\includegraphics[scale=0.9]{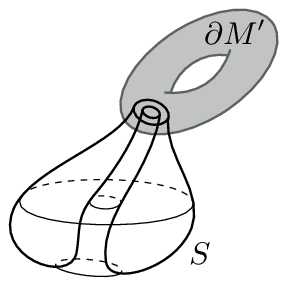}\qquad}
        \qquad\qquad
        \subfigure[Three boundary components]{\label{fig-ann3}%
            \quad\includegraphics[scale=0.9]{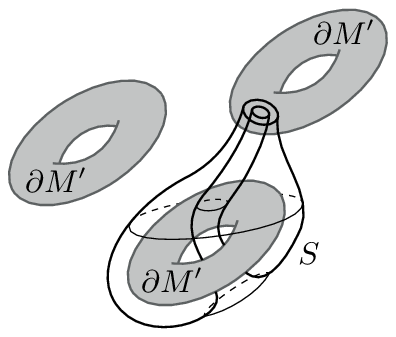}\quad}
        \caption{Resolving cases (\ref{en-test-ann1}) and
            (\ref{en-test-ann3}) of Test~\ref{test-annuli}}
        \end{figure}

        \item \emph{Case~\ref{en-test-ann3}:}
        Suppose $\mfdb'$ has three boundary components
        and $S$ splits $\mfdb'$ into a pair of $\mathrm{Torus} \times I$
        pieces.
        The only way to arrange this is for the two trivial curves of
        $\bdry S$ to sit one inside the other on some component of
        $\bdry \mfdb'$, as shown in Figure~\ref{fig-ann3}.
        Now $\mfdb'$ is now obtained by attaching the two
        $\mathrm{Torus} \times I$ pieces along an annulus that is
        trivial on one $\mathrm{Torus} \times I$ boundary
        but non-trivial on the other;
        this makes $\mfdb'$ a connected sum of the solid torus with
        $\mathrm{Torus} \times I$, contradicting hyperbolicity once more.
        \qedhere
    \end{itemize}
\end{proof}

Our final test for non-hyperbolicity is a special case, but one for
which we have an algorithm that is always conclusive and
(despite requiring theoretical exponential time) fast in practice:

\begin{test}[Solid torus recognition] \label{test-solidtorus}
    Let $\tri$ be an ideal 3-manifold triangulation.
    Truncate the ideal vertices of $\tri$, and switch to the orientable
    double cover if $\tri$ is non-orientable.
    If the resulting triangulation represents the solid torus, then
    $\tri$ does not represent a {\hypmfd}.
\end{test}

Solid torus recognition features repeatedly throughout this suite
of tests (see
Tests~\ref{test-torikb}, \ref{test-annuli} and \ref{test-solidtorus}),
and so we outline the algorithm below.  This algorithm does not appear in
the literature, but it is a straightforward assembly of well-known components:
the overall crushing framework is due to Jaco and Rubinstein
\cite{jaco03-0-efficiency}, and the branch-and-bound search for normal
discs and spheres is due to the author and Ozlen \cite{burton12-unknot}.

\begin{algorithm} \label{a-solidtorus}
    Let $\tri$ be a 3-manifold triangulation with boundary triangles
    (i.e., not an ideal triangulation).  To test whether or not
    $\tri$ represents the solid torus:
    \begin{itemize}
        \item Check that $\tri$ is orientable with $H_1=\Z$ and
        a single torus boundary component.  If not, then $\tri$ is not
        the solid torus.

        \item Repeatedly search for non-vertex-linking normal surfaces
        of positive Euler characteristic, and crush them as described in
        Section~\ref{s-prelim-crushing}.
        Then $\tri$ represents the solid torus if and only if,
        after crushing, the resulting triangulation is a
        (possibly empty) disjoint union of 3-spheres and/or 3-balls.
    \end{itemize}
\end{algorithm}

The bottleneck of this algorithm is the search for non-vertex-linking
normal surfaces of positive Euler characteristic (although 3-sphere
and 3-ball recognition are non-trivial, they too reduce to this
same bottleneck).
For this step we use the branch-and-bound framework described in
\cite{burton12-unknot}: essentially we work through a hierarchy of
linear programs that, in many experiments, is seen to prune the
exponential search tree down to linear size.

\begin{proof}[Proof for Algorithm~\ref{a-solidtorus}]
    The arguments are standard, and so we merely sketch the proof here.

    The algorithm terminates because each crushing operation reduces the
    total number of tetrahedra.
    Correctness follows from two
    observations: (i)~that any solid torus must contain a
    normal meridional disc \cite{haken61-knot}; and (ii)~when crushing a disc
    or sphere in an orientable and non-ideal triangulation,
    the only possible topological
    side-effects (beyond cutting along and collapsing the original
    disc or sphere) are cutting along additional discs or spheres,
    filling boundary spheres with balls,
    and/or deleting entire $B^3$, $S^3$, $\R P^3$ or $L_{3,1}$
    components \cite{jaco03-0-efficiency}.

    It follows that repeatedly crushing spheres and discs in the solid torus
    \emph{must} leave us with a disjoint union of 3-spheres and/or 3-balls,
    but if we start with some different manifold with torus boundary
    and $\Z$ homology then we must have some different topological component
    left over.
\end{proof}


\subsection{Hyperbolic manifolds} \label{s-test-hyp}

This concludes our suite of non-hyperbolicity and non-minimality tests;
all that remains is to certify that the remaining manifolds
are hyperbolic and (for $n \leq 8$ tetrahedra) to locate them in the
Callahan-Hildebrand-Thistlethwaite-Weeks census.
For this we call upon pre-existing software:
we ask {\snappea} to find hyperbolic structures and {\hikmot} to
rigorously verify them.

To account for the fact that we may encounter non-geometric
triangulations of hyperbolic manifolds, we call upon {\snappea}'s
built-in ``randomisation'' routine, which randomly performs a sequence
of local moves (e.g., Pachner moves) to obtain a different
triangulation of the same manifold (hopefully one that is geometric).

Our hyperbolicity test is simple: we identify all of
those triangulations that we \emph{suspect} to be hyperbolic (using
{\snappea}), group these into classes of triangulations that
represent the same manifold (using their Epstein-Penner decompositions),
and use {\hikmot} to rigorously certify that at least one triangulation
in each class is geometric.  The details are as follows.

\begin{test}[Hyperbolicity test with $r$ randomisations] \label{test-hyp}
    Let $\tri$ be an ideal 3-manifold triangulation, and let $r \in \N$.
    We ask {\snappea} to find a complete hyperbolic structure on $\tri$.
    If this fails, we make up to $r$ additional attempts in which we ask
    {\snappea} to randomise the triangulation and try again.

    If at any attempt {\snappea} does find a
    complete hyperbolic structure then we stop,
    ask {\snappea} to compute the corresponding Epstein-Penner
    decomposition $\mathcal{D}$,
    and ask {\hikmot} to attempt to rigorously verify that the
    corresponding triangulation is geometric.\footnote{%
        For non-orientable triangulations, we must run {\hikmot} over
        the orientable double cover.}
    We then record the original input triangulation $\tri$,
    the Epstein-Penner decomposition $\mathcal{D}$,
    and whether {\hikmot}'s verification was successful.
    If all attempts to find a complete hyperbolic structure fail,
    then we declare the test inconclusive for $\tri$.

    For each Epstein-Penner decomposition $\mathcal{D}$ that we record,
    let $\tri_1,\ldots,\tri_k$ be the
    (possibly many) input triangulations from which it was obtained.
    If {\hikmot}'s verification succeeded for \underline{any} of these
    $\tri_i$, then we declare that \underline{all} of these $\tri_i$
    represent a hyperbolic manifold.
    Otherwise the test is inconclusive for all of $\tri_1,\ldots,\tri_k$.
\end{test}

We recall that {\snappea}'s complete hyperbolic structures are
merely approximations computed using inexact floating-point arithmetic,
and that its Epstein-Penner decompositions may or may not be correct.
Nevertheless, we can prove that our declarations of hyperbolicity
are rigorous:

\begin{proof}
    Consider some class $\tri_1,\ldots,\tri_k$ for which {\snappea}
    computes a common Epstein-Penner decomposition $\mathcal{D}$,
    and for which {\hikmot} certifies some $\tri_i$ to be geometric.
    Although {\snappea} might not have computed the Epstein-Penner
    decompositions correctly, it nevertheless computes these decompositions
    using Pachner moves \cite{weeks93-convex}, and so we can still
    guarantee that $\tri_1,\ldots,\tri_k$ represent the same manifold.
    It follows that, since {\hikmot} has certified hyperbolicity for
    one of the $\tri_i$, then all of the $\tri_i$ represent
    this same hyperbolic manifold.
\end{proof}

Note that {\snappea} itself has a mechanism to automatically
randomise triangulations (up to $64$ times) when attempting to find
a complete hyperbolic structure.
We disable this feature here, and instead allow a much smaller number
of retriangulations (by default, $r=8$).  This is to keep the test
reasonably fast for triangulations that represent non-hyperbolic manifolds
(the vast majority of cases), for which {\snappea} will fail on all
$r$ retriangulations and the test will always be inconclusive.


%
%

\section{Building and processing the census} \label{s-census}

Here we briefly outline how we organise the computations, deal with
the handful of unresolved cases, and finalise the proof of our main results
(Theorem~\ref{t-falseneg}, Corollary~\ref{c-complete8} and
Theorem~\ref{t-complete9}).

\begin{table}[tb]
\footnotesize
\caption{Running the tests of Section~\ref{s-process} over
all enumerated triangulations} \label{tab-tests}
\begin{tabular}{l|rrrrrrrrr}
\# Tetrahedra & 1 & 2 & 3 & 4 & 5 & 6 & 7 & 8 & 9\\
\hline
Test~\ref{test-greedy}: Greedy non- & & & & & & & & \\
\quad minimality test &
$-$ & $-$ & $-$ & 25 & 268 & 2\,224 & 16\,622 & 130\,951 & 902\,439\\
Test~\ref{test-hyp}: Hyperbolicity & & & & & & & & \\
\quad (no randomisations) &
1 & 5 & 21 & 136 & 548 & 2\,647 & 11\,341 & 48\,437 & 204\,139\\
Test~\ref{test-eff}: Spheres and & & & & & & & & \\
\quad projective planes &
$-$ & $-$ & $-$ & $-$ & $-$ & $-$ & 17 & 203 & 1\,430\\
Test~\ref{test-hyp}: Hyperbolicity & & & & & & & & \\
\quad ($r=8$ randomisations) &
$-$ & $-$ & $-$ & $-$ & 4 & 52 & 360 & 1\,974 & 11\,608\\
Test~\ref{test-torikb}: Tori and Klein & & & & & & & & \\
\quad bottles &
$-$ & 2 & 10 & 63 & 253 & 1\,379 & 6\,666 & 35\,338 & 186\,104\\
Test~\ref{test-comb}: Combinatorial & & & & & & & & \\
\quad recognition &
$-$ & $-$ & $-$ & $-$ & $-$ & 6 & 10 & 75 & 417\\
Test~\ref{test-exhaustive}: Exhaustive non- & & & & & & & & \\
\quad minimality ($h=2$) &
$-$ & $-$ & $-$ & $-$ & 2 & 40 & 293 & 1\,427 & 6\,195\\
Test~\ref{test-annuli}: Annuli and & & & & & & & & \\
\quad {\mb}s &
$-$ & $-$ & $-$ & $-$ & $-$ & $-$ & $-$ & 32 & 366\\
Test~\ref{test-solidtorus}: Solid torus & & & & & & & & \\
\quad recognition &
$-$ & $-$ & $-$ & $-$ & $-$ & $-$ & $-$ & $-$ & $-$\\
\hline
Unresolved & $-$ & $-$ & $-$ & $-$ & $-$ & $-$ & 3 & 39 & 354\\
\hline
Total &
1 & 7 & 31 & 224 & 1\,075 & 6\,348 & 35\,312 & 218\,476 & 1\,313\,052
\end{tabular}
\end{table}

For each of the $1\,574\,526$ triangulations that we enumerate in
Theorem~\ref{t-enumerate}, we run the tests of Section~\ref{s-process} in
the order shown in Table~\ref{tab-tests}.  This table also shows how
many cases each test resolves.  Once a triangulation is certified as either
(i)~non-minimal and/or non-hyperbolic or (ii)~hyperbolic, then we do not
run any further tests on it (i.e., we expect the numbers towards the
bottom of Table~\ref{tab-tests} to be small, since very few cases remain
by that stage).

As a rough guide, the tests are ordered so that (i)~faster tests are
placed before slower tests; and (ii)~tests that are likely to resolve
many cases are placed earlier.  For example, the polynomial-time greedy
non-minimality test is placed well before the super-exponential-time
exhaustive non-minimality test, and tests that require building a
truncated orientable double cover (which could increase the number of
tetrahedra significantly) are placed towards the end.  Solid torus
recognition is placed last, since in practice one finds that most
triangulations of the solid torus are easy to simplify, and are
therefore caught by the polynomial-time greedy non-minimality
and/or combinatorial recognition tests.

As seen in Table~\ref{tab-tests}, the full suite of tests leaves just
$396$ of the $1\,574\,526$ triangulations unresolved.  For these $396$
cases:
\begin{itemize}
    \item 4~were cases that crashed {\snappea} during the initial run.
    Rerunning them (with a different random seed) shows them to be
    hyperbolic (Test~\ref{test-hyp}).
    \item The remaining cases are resolved by an exhaustive
    search through different triangulations of the same manifold.
    As in Test~\ref{test-exhaustive},
    we try all possible combinations of 2-3 and 3-2 Pachner moves
    without exceeding $h$ additional tetrahedra.
    For $h=3$, this certifies 175 as hyperbolic (by producing a
    new triangulation that we had already certified as such),
    and 166 as non-minimal (by producing a new triangulation with fewer
    tetrahedra).  Of the 51 leftover cases, rerunning with $h=4$ certifies
    1 as hyperbolic and the final 50 as non-minimal.
\end{itemize}

The overall result is that $281\,453$ of the original $1\,574\,526$
triangulations are certified to represent {\hypmfd}s,
and all others are certified as non-minimal and/or
non-hyperbolic.  At this stage we partition our triangulations into classes
that represent the same manifold:
\begin{itemize}
    \item
    Test~\ref{test-hyp} records the Epstein-Penner decomposition
    computed by {\snappea} for each triangulation.
    Although these Epstein-Penner decompositions might be incorrect,
    it is guaranteed that triangulations that produce the same cell
    decomposition represent the same manifold, and so we can group them
    together as such.
    This produces exactly $76\,000$ manifold classes.

    \item
    From each triangulation we run yet another exhaustive search through
    2-3 and 3-2 Pachner moves (without exceeding $h=2$ additional
    tetrahedra) in order to identify classes that should be
    merged---this is necessary because {\snappea} does indeed compute some
    Epstein-Penner decompositions incorrectly.
    This reduces our $76\,000$ manifold classes down to the final
    list of $75\,956$.

    \item
    We discard any triangulation that does not use the fewest tetrahedra
    in its class (since these are clearly non-minimal).
    This cuts our list of $281\,453$ hyperbolic triangulations
    down to the final list of $229\,112$ triangulations.
\end{itemize}

We emphasise that there is still no guarantee that our $75\,956$
manifold classes represent distinct manifolds, and therefore that our
$229\,112$ triangulations are indeed minimal.  This we resolve
later in Section~\ref{s-dup} using algebraic invariants.

This concludes the proof of Theorem~\ref{t-complete9}
(that our new 9-tetrahedron census has no intruders and no omissions).
To finish the proofs of Theorem~\ref{t-falseneg} and
Corollary~\ref{c-complete8} (that the older
Callahan-Hildebrand-Thistlethwaite-Weeks
census has no intruders and no omissions), we simply verify that
(i)~each manifold class with an $n\leq 8$-tetrahedron triangulation
includes a representative from the
Callahan-Hildebrand-Thistlethwaite-Weeks census; and that
(ii)~each manifold from this census is contained in at least one such class.

\begin{figure}[tb]
    \includegraphics[scale=0.6]{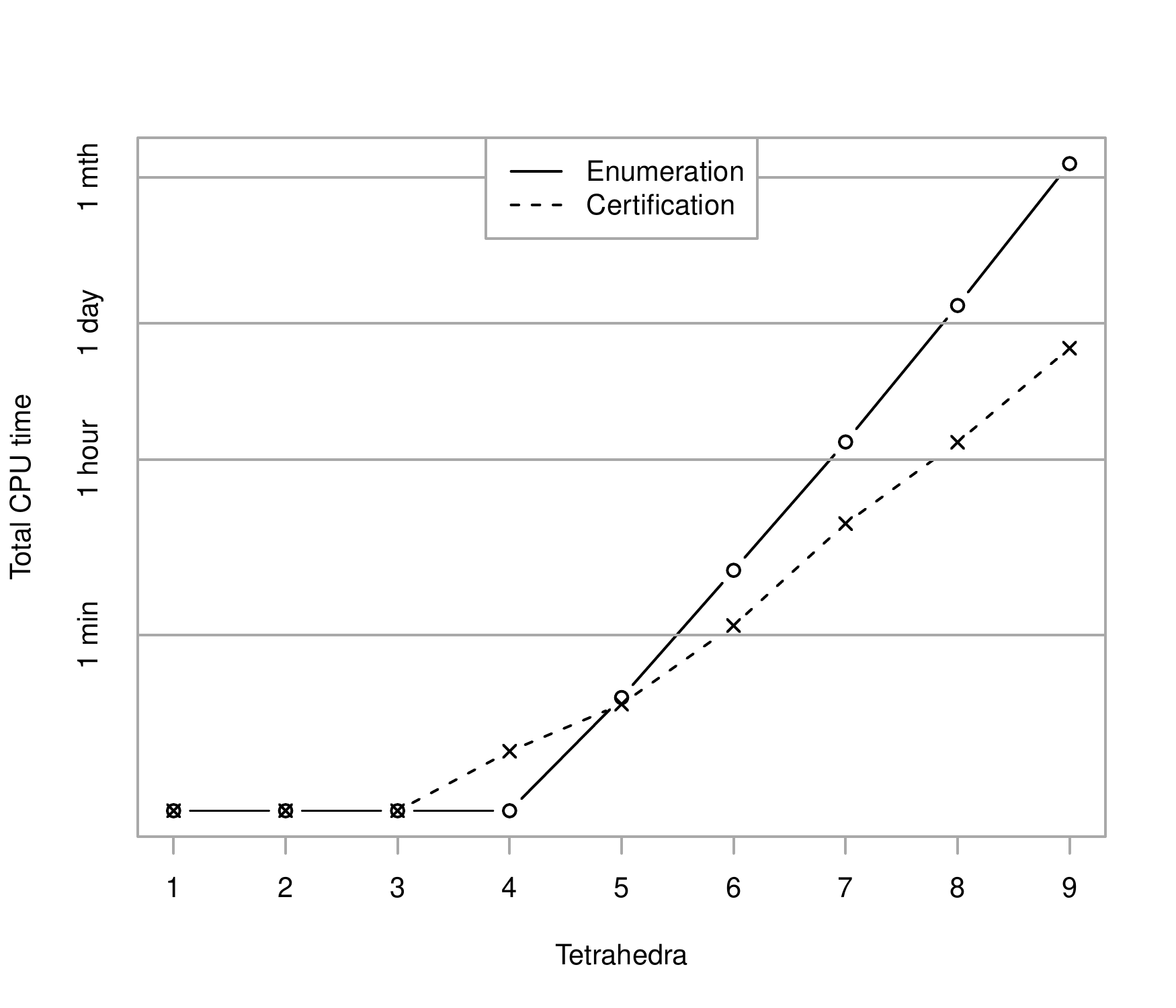}
    \caption{Running times}
    \label{fig-times}
\end{figure}

Figure~\ref{fig-times} plots the total running times for both the
enumeration of our $1\,574\,526$ triangulations as described in
Section~\ref{s-gen}, and the certification of triangulations as
non-minimal, non-hyperbolic and/or hyperbolic
using the tests of Section~\ref{s-process}.
As seen in the plot, the certification time---though significant---is a
mere fraction of the time required for enumeration.

This backs up our claim that the tests are fast---the
bottleneck for producing a rigorous and verified cusped hyperbolic
census is not the verification, but just the raw generation of candidate
triangulations.  Looking forward, it therefore seems
entirely feasible for future extensions of the cusped hyperbolic census to
be made rigorous in this way.

%
%

\section{Preventing duplicates} \label{s-dup}

To finish, we describe how we prove Theorem~\ref{t-dup}---that
no two manifolds in the census are homeomorphic.
For this we use algebraic invariants, which can be computed and
compared exactly.

\begin{notation}
    Let $G$ be any group.
    We let $\ab{G}$ denote the abelianisation of $G$.
    For any finite index $i \in \N$,
    we let $S_i(G)$ denote a set containing one arbitrary representative
    of each conjugacy class of index $i$ subgroups of $G$.
    Finally, we let $\abi{i}{G}$ denote the multiset of abelianisations
    $\{ \ab{s}\,|\,s \in S_i(G)\}$;
    here \emph{multiset} means that we record multiple occurrences of the
    same abelian group.
\end{notation}

Although $S_i(G)$ is not well-defined (since we have a choice of
representatives), the multiset of abelianisations $\abi{i}{G}$ is.
If $G$ is given as a finite presentation then $\abi{i}{G}$ is a
computable invariant of $G$, though this computation is
only feasible for small indices $i$ since the running time can grow
exponentially (or worse) with $i$ \cite{sims94-computation}.

\begin{lemma} \label{l-magma}
    Amongst all {\hypmfd}s that can be triangulated with $n \leq 9$
    ideal tetrahedra, any two distinct manifolds
    differ in at least one of the following invariants:
    \begin{itemize}
        \item orientability;
        \item the first homology $H_1 = \ab{\pi_1}$;
        \item for each index $2 \leq i \leq 11$,
        the multiset of abelianisations $\abi{i}{\pi_1}$.
    \end{itemize}
    Here $\pi_1$ denotes, as usual, the fundamental group of a 3-manifold.
\end{lemma}

\begin{proof}
    By Theorem~\ref{t-complete9}, our $75\,956$ census manifolds
    together contain all {\hypmfd}s that can be triangulated with $n \leq 9$
    ideal tetrahedra (though we do not yet know whether these manifolds are
    distinct).  We therefore prove Lemma~\ref{l-magma} simply by
    computing the invariants above for all $75\,956$ census manifolds
    and observing that they partition the manifolds into
    $75\,956$ distinct homeomorphism classes.
\end{proof}

Theorem~\ref{t-dup} now follows immediately from the computations in the
proof above.

\medskip

For the computations themselves we use {\magma} \cite{bosma97-magma},
which is fast and robust enough to compute hundreds of thousands of
such invariants within a reasonable running time.
(Other software packages such as {\snappy} and {\gap} can also compute
these invariants, but they are infeasibly slow for some problematic cases.)

\begin{figure}[tb]
    \includegraphics[scale=0.6]{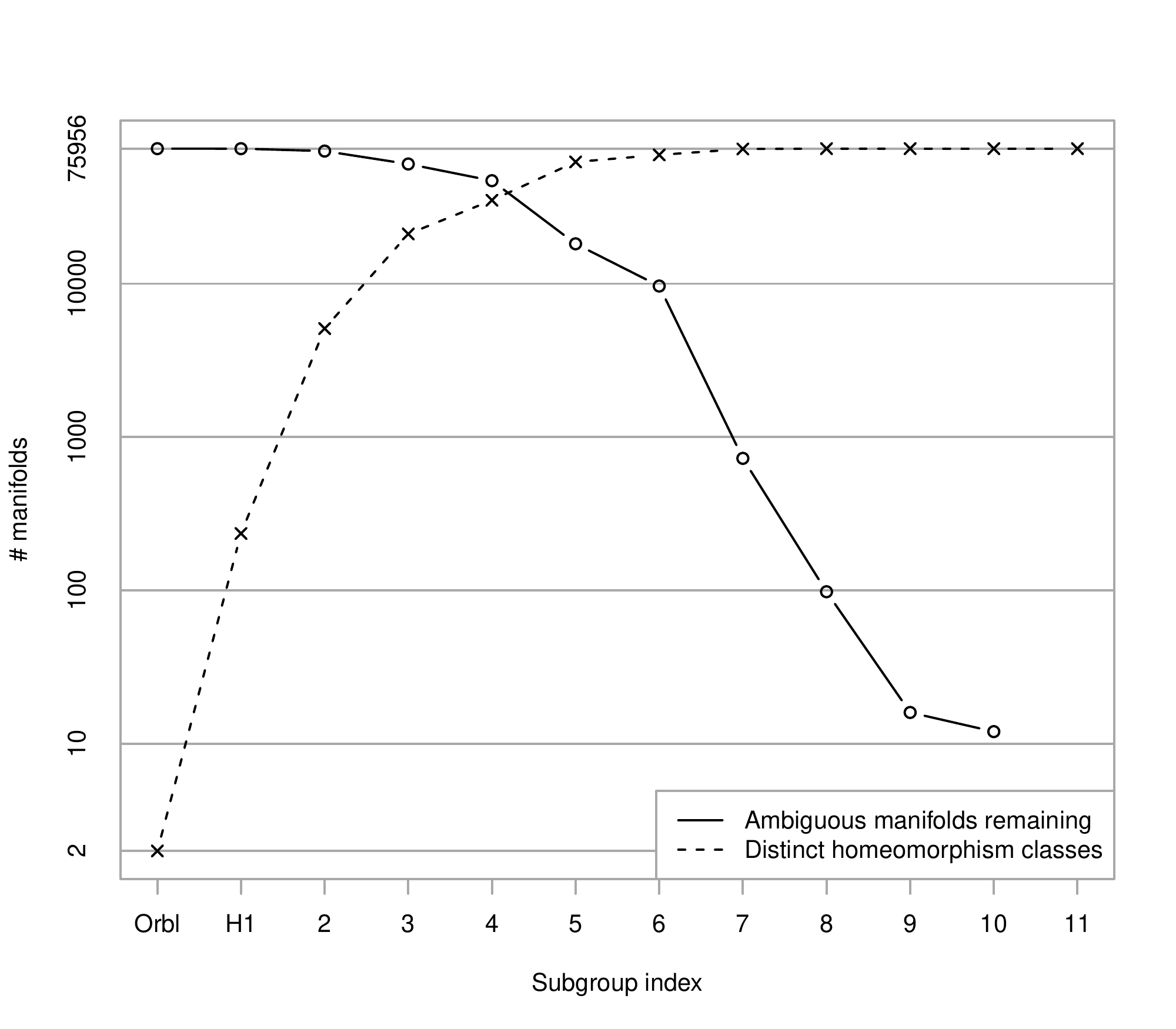}
    \caption{Proving that all $75\,956$ census manifolds are distinct}
    \label{fig-classes}
\end{figure}

Of course we do not need to compute all such invariants for all
$75\,956$ manifolds.
We compute $H_1$, $\abi{2}{\pi_1}$, \ldots, $\abi{11}{\pi_1}$ in 11
distinct stages (these are ordered by increasing difficulty),
and after each stage we put aside those census manifolds that have
been successfully distinguished from all others.
Those manifolds that remain (i.e., whose invariants so far coincide with
at least one other manifold) we call \emph{ambiguous manifolds}, and we
keep these for further processing.
Figure~\ref{fig-classes} plots how many ambiguous manifolds remain
after computing $\abi{i}{\pi_1}$ for each subgroup index $i$.
Note that the vertical axis uses a logarithmic scale.

\begin{figure}[tb]
    \includegraphics[scale=0.6]{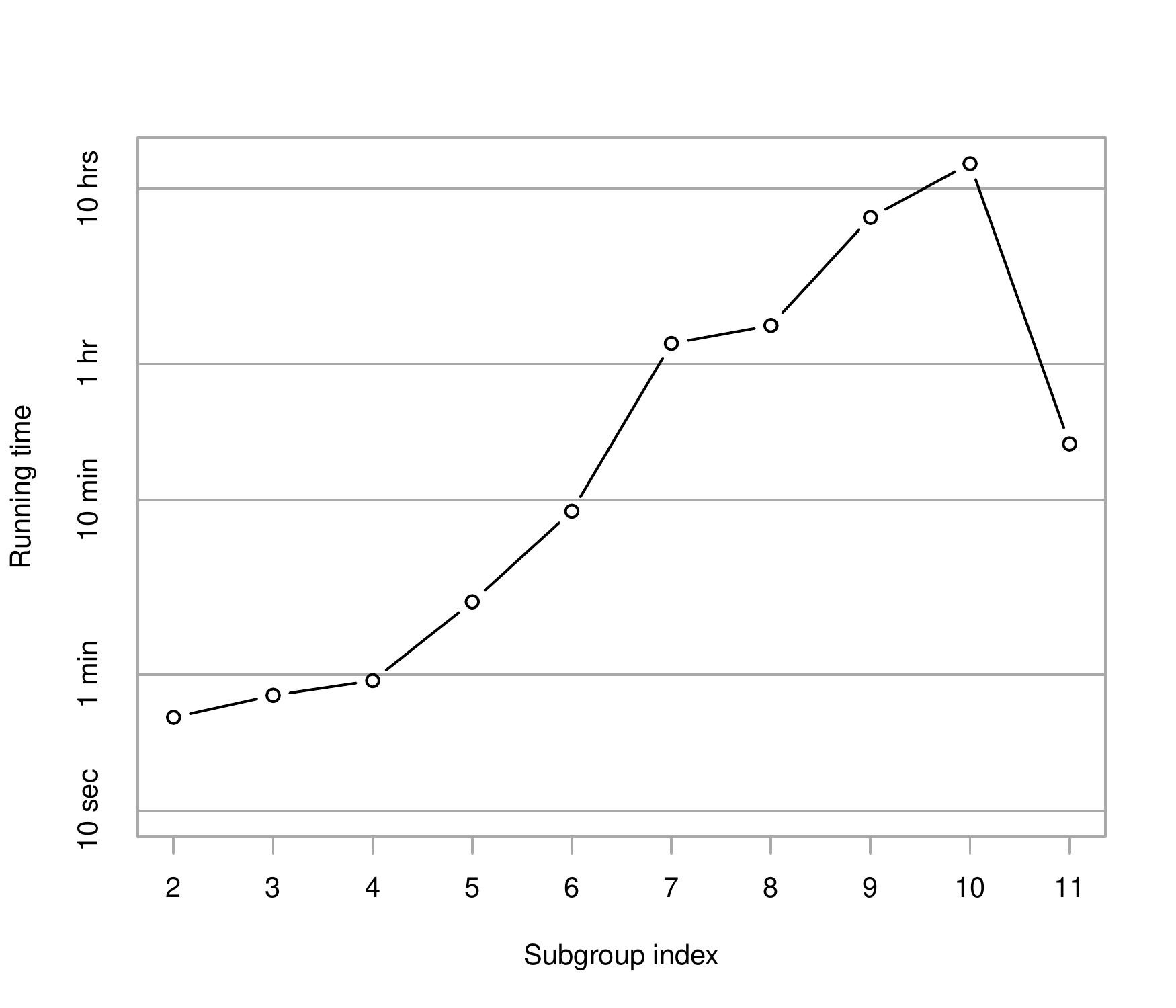}
    \caption{Running times for computing low-index subgroups}
    \label{fig-magma-times}
\end{figure}

Figure~\ref{fig-magma-times} plots the running time required to compute
$\abi{i}{\pi_1}$ for each $i$, again on a logarithmic scale
(measured on a single 3.47GHz Intel Xeon CPU).
This running time was only for ambiguous manifolds, as explained above;
in particular, the peak time of $14$~hours for index $i=10$ was for
just $16$ manifolds that remained at that late stage of processing.
We are perhaps fortunate that $i=11$ was sufficient to distinguish all
of the census manifolds, and that we were not required to increase the
index any further.

We note that these final $16$ manifolds are difficult to distinguish
even with the help of inexact floating-point invariants.
They come in eight ambiguous
pairs of non-orientable manifolds, where both manifolds in each pair have
the same hyperbolic volume and shortest geodesic, and indeed
homeomorphic orientable double covers.
Examining a larger section of the length spectrum does, however, reveal
enough information to distinguish them (subject to numerical error).
For reference, these final eight pairs appear in the census tables as:
\[ \small \begin{array}{llll}
   (\mathtt{y296},\  \mathtt{y297}) &
   (\mathtt{n8\_0534},\  \mathtt{n8\_0535}) &
   (\mathtt{n9\_1033},\  \mathtt{n9\_1034}) &
   (\mathtt{n9\_1091},\  \mathtt{n9\_1092}) \\
   (\mathtt{n9\_1183},\  \mathtt{n9\_1184}) &
   (\mathtt{n9\_1200},\  \mathtt{n9\_1201}) &
   (\mathtt{n9\_2015},\  \mathtt{n9\_2016}) &
   (\mathtt{n9\_2121},\  \mathtt{n9\_2122}).
   \end{array} \]

%
%

\bibliographystyle{amsplain}
\bibliography{pure}

\end{document}